\documentclass[12pt,leqno]{amsart}
\usepackage{amssymb}
\usepackage{t1enc}
\usepackage[mathscr]{eucal}
\usepackage{parskip}
\usepackage{xcolor}
\usepackage{tikz}
\usepackage{tikz-cd}
\usepackage{hyperref}
\numberwithin{equation}{section}
\usepackage[margin=2.9cm]{geometry}
\usepackage{booktabs}
\usepackage{enumerate}
\usepackage{bm}
\usepackage{xfrac}

\usepackage{listings}
\DeclareFontFamily{U}{mathb}{\hyphenchar\font45}
\DeclareFontShape{U}{mathb}{m}{n}{
  <-5.5> mathb5
  <5.5-6.5> mathb6
  <6.5-7.5> mathb7
  <7.5-8.5> mathb8
  <8.5-9.5> mathb9
  <9.5-11> mathb10
  <11-> mathb12
}{}
\DeclareSymbolFont{mathb}{U}{mathb}{m}{n}
\DeclareFontSubstitution{U}{mathb}{m}{n}
\DeclareMathSymbol{\bigast}        {\mathop}{mathb}{"06}

\DeclareFontFamily{U}{mathx}{\hyphenchar\font45}
\DeclareFontShape{U}{mathx}{m}{n}{
<-6> mathx5 <6-7> mathx6 <7-8> matha7
<8-9> mathx8 <9-10> mathx9
<10-12> mathx10 <12-> mathx12
}{}
\DeclareSymbolFont{mathx}{U}{mathx}{m}{n}
\DeclareMathSymbol{\bigtimes}{\mathop}{mathx}{"91}

\theoremstyle{plain}
\newtheorem{thm}{Theorem}[section]
\newtheorem{lem}[thm]{Lemma}
\newtheorem{cor}[thm]{Corollary}
\newtheorem{prop}[thm]{Proposition}

 \theoremstyle{definition}
\newtheorem{defn}[thm]{Definition}
\newtheorem{rem}[thm]{Remark}
\newtheorem{ex}[thm]{Example}
\newtheorem{notn}[thm]{Notation}
\newtheorem{setup}[thm]{Setup}

\newcommand{\mb}[1]{\mathbb{#1}}
\newcommand{\mc}[1]{\mathcal{#1}}

\newcommand{\mr}[1]{\mathrm{#1}}

\newcommand{\vphi}{\varphi}
\newcommand{\Spec}{\operatorname{Spec}}
\newcommand{\GW}{\mathrm{GW}}

\newcommand{\rank}{\operatorname{rank}}
\newcommand{\disc}{\operatorname{disc}}

\newcommand{\nplus}{\oplus^\mathrm{N}}
\newcommand{\colim}{\operatorname{colim}}

\newcommand{\Th}{\mathrm{Th}}
\newcommand{\Sm}{\mathrm{Sm}}
\newcommand{\cv}{\curlyvee}
\newcommand{\Kmw}{\bm{\mathrm{K}}^{\mathrm{MW}}}
\renewcommand{\P}{\mathbb{P}}
\newcommand{\G}{\mathbb{G}}
\newcommand{\A}{\mathbb{A}}

\renewcommand{\H}{\mathcal{H}_{\bullet}}

\newcommand{\Bez}{\mathrm{B{\acute e}z}}

\newcommand{\commu}{\mathrm{com}}
\newcommand{\mon}{\mathrm{mon}}

\newcommand{\Nwt}{\mathrm{Nwt}}

\newcommand{\slant}{\mathrm{slant}}

\mathchardef\mhyphen="2D
\newcommand{\tnil}{\mathrm{2\mhyphen nil}}

\newcommand{\alg}{\mathrm{alg}}
\newcommand{\ab}{\mathrm{ab}}
\newcommand{\Res}{\operatorname{Res}}
\newcommand{\D}{\mathfrak{D}}
\newcommand{\Jac}{\operatorname{Jac}}
\newcommand{\red}{\mathrm{red}}
\newcommand{\pA}{\pi_1^{\mathbb{A}^1}}

\makeatletter
\@namedef{subjclassname@2020}{\textup{2020} Mathematics Subject Classification}
\makeatother

\begin{document}
\title{Motivic configurations on the line}

\author[Igieobo]{John Igieobo}
\address{Georgia Institute of Technology} 
\email{oririigieobo@gmail.com}

\author[McKean]{Stephen McKean}
\address{Department of Mathematics \\ Brigham Young University} 
\email{mckean@math.byu.edu}
\urladdr{shmckean.github.io}

\author[Sanchez]{Steven Sanchez}
\address{Georgia State University} 
\email{steve122900@gmail.com}

\author[Taylor]{Dae'Shawn Taylor}
\address{Georgia Institute of Technology} 
\email{daeshawntaylor@gmail.com}

\author[Wickelgren]{Kirsten Wickelgren}
\address{Department of Mathematics \\ Duke University}
\email{kirsten.wickelgren@duke.edu}
\urladdr{services.math.duke.edu/~kgw}


\begin{abstract}
For each configuration of rational points on the affine line, we define an operation on the group of unstable motivic homotopy classes of endomorphisms of the projective line. We also derive an algebraic formula for the image of such an operation under Cazanave and Morel's unstable degree map, which is valued in an extension of the Grothendieck--Witt group. In contrast to the topological setting, these operations depend on the choice of configuration of points via a discriminant. We prove this by first showing a local-to-global formula for the global unstable degree as a modified sum of local terms. We then use an anabelian argument to generalize from the case of local degrees of a global rational function to the case of an arbitrary collection of endomorphisms of the projective line.
\end{abstract}

\maketitle

\section{Introduction}\label{sec:introduction}
In topology, May's recognition principle characterizes loop spaces as algebras over the little cubes operad \cite{May72}, which is defined by operations coming from configuration spaces of Euclidean space. An analog of May's recognition principle for $\mb{P}^1$-loop spaces in unstable motivic homotopy theory has been sought for the last quarter century. We offer some thoughts on this question by defining a family of operations $\sum_D$ on the $\mb{P}^1$-loop space $\Omega_{\mb{P}^1}\mb{P}^1$. We construct these operations in terms of the configuration space of rational points in the affine line --- indeed, the subscript $D$ refers to such a configuration of points. In contrast to the topological setting, the homotopy classes of these operations depend on the set of points $D$ via a sort of discriminant.

Let $k$ be a field, and let $D=\{r_1,\ldots,r_n\}$ be a subset of $\mb{A}^1_k(k)$ with $r_i\neq r_j$ for $i\neq j$. We define the \emph{$D$-pinch map} (see Definition~\ref{def:D-pinch}) as the composite
\[\cv_D:\mb{P}^1_k\xrightarrow{c_D}\frac{\mb{P}^1_k}{\mb{P}^1_k-D}\xrightarrow{\cong}\bigvee_{i=1}^n\frac{\mb{P}^1_k}{\mb{P}^1_k-\{r_i\}}\xleftarrow{\simeq}\bigvee_{i=1}^n\mb{P}^1_k,\]
where $c_D$ is the collapse map induced by the inclusion $\mb{P}^1_k-D\hookrightarrow\mb{P}^1_k$, the second map is a canonical isomorphism of motivic spaces resulting from purity, and the last equivalence is a wedge of collapse maps coming from the inclusions $\mb{P}^1_k-\{r_i\}\hookrightarrow\mb{P}^1_k$. For endomorphisms $f_1,\ldots,f_n:\mb{P}^1_k\to\mb{P}^1_k$ in the unstable motivic homotopy category, we define the \emph{$D$-sum} (see Definition~\ref{def:D-sum}) to be
\[\sum_D(f_1,\ldots,f_n):=\vee_i f_i\circ\cv_D.\]

Note that $\sum_D(f_1,\ldots,f_n)$ is again an endomorphism of the motivic space $\mb{P}^1_k$. Morel proved that such endomorphisms can be understood in terms of quadratic forms: he defined an analog of the Brouwer degree~\cite{Mor06}, which is a ring homomorphism from the ring of endomorphisms of the sphere $S^n\wedge\mb{G}_m^{\wedge n}\simeq\mb{P}^n_k/\mb{P}^{n-1}_k$ to the Grothendieck--Witt ring $\GW(k)$ of isomorphism classes of non-degenerate symmetric bilinear forms over a field $k$. In dimensions 2 and greater, Morel's degree map is an isomorphism. In dimension 1, the degree is surjective but not injective. Morel~\cite[Theorem 7.36]{Mor12} also computed
\begin{align}\label{eq:cazanave}
    [\mb{P}^1_k,\mb{P}^1_k]\cong\GW(k)\times_{k^\times/(k^\times)^2}k^\times,
\end{align}
and Cazanave~\cite{Caz12} gave an explicit formula for this isomorphism. Let $\GW^u(k):=\GW(k)\times_{k^\times/(k^\times)^2}k^\times$, which we call the \textit{unstable Grothendieck--Witt group}. Let
\[\deg^u:[\mb{P}^1_k,\mb{P}^1_k]\to\GW^u(k)\]
denote the \emph{unstable degree}. Our main theorem is a characterization of the $D$-sum in terms of its image under $\deg^u$.

\begin{thm}\label{thm:D-sum-all-f_i}
Let $D=\{r_1,\ldots,r_n\}\subset\mb{A}^1_k(k)$. For any unstable pointed $\mb{A}^1$-homotopy classes of maps $f_1,\ldots,f_n\in[\mathbb{P}^1,\mathbb{P}^1]$, we have 
\[
\deg^u (\sum_D (f_1,\ldots, f_n)) = \big(\bigoplus_{i=1}^n\beta_i,\prod_{i=1}^n d_i\cdot\prod_{i<j}(r_i-r_j)^{2m_im_j}\big),
\] where $(\beta_i, d_i)=\deg^u(f_i)$ and $m_i = \rank \deg^u(f_i)$ for each $i$.
\end{thm}

The proof of Theorem~\ref{thm:D-sum-all-f_i} proceeds in two steps. The first step is to give a local-to-global formula for the unstable $\mb{A}^1$-degree of a rational function. To this end, we develop an unstable analog of the local $\mb{A}^1$-degree \cite{KW19} and apply algebraic methods due to Cazanave \cite{Caz12}. As a result, we find that Theorem~\ref{thm:D-sum-all-f_i} holds when $f_1,\ldots,f_n$ represent the unstable local degrees of a rational function whose vanishing locus is $\{r_1,\ldots,r_n\}$.

\begin{thm}\label{thm:D-sum-local}
Let $f/g$ be a pointed rational function with vanishing locus $\{r_1,\ldots,r_n\}\subset\mb{A}^1_k(k)$. For each $i$, let $\deg^u_{r_i}(f/g)=(\beta_i,d_i)$ and $\rank\beta_i=m_i$. Then
\begin{equation}\label{eq:intro-ph-formula}
\deg^u(f/g)=\big(\bigoplus_{i=1}^n\beta_i,\prod_{i=1}^n d_i\cdot\prod_{i<j}(r_i-r_j)^{2m_im_j}\big).
\end{equation}
\end{thm}

Theorem~\ref{thm:D-sum-local} will serve as the base case of an induction argument for Theorem~\ref{thm:D-sum-all-f_i}. While carrying out this first step, we prove a few results of some independent interest; we will mention these momentarily. 

The second step to proving Theorem~\ref{thm:D-sum-all-f_i} is an inductive argument that uses results of Morel on the fundamental group sheaf $\pA(\mb{P}^1)$. Morel showed that $\mb{P}^1$ is \emph{anabelian} in $\mb{A}^1$-homotopy theory \cite[Remark~7.32]{Mor12}, in the sense that the $\mb{A}^1$-fundamental group yields a group isomorphism
\begin{equation}\label{MorelThmP1Anabelian}
[\mb{P}^1_k,\mb{P}^1_k]\cong\mr{End}(\pA(\mb{P}^1_k)).\end{equation}
Here, we borrow the term \emph{anabelian} from Grothendieck's anabelian program on the \'etale fundamental group \cite{Gro97}. Morel's anabelian theorem implies that for Theorem~\ref{thm:D-sum-all-f_i}, it suffices to prove the analogous result after applying $\pA$.

The map $\pA(\cv_D)$ has target $\pA(\bigvee_{i=1}^n \P^1)$, which Morel shows is the initial strongly $\A^1$-invariant sheaf on the free product. While maps to this sheaf seem to us to be difficult to control, Morel shows that $\pA(\P^1)$ is $2$-nilpotent. Thus, we are free to pass to various sorts of $2$-nilpotent quotients of $\pA(\vee_{i=1}^n \P^1)$ while computing $\pA(\vee f_i \circ \cv_D)$. There are subtleties involving $\A^1$-invariance and nilpotence (c.f.~\cite{Asok_Bac_Hop24}). 

We define a $2$-nilpotent quotient for our purposes (Lemma \ref{lem:Kr_central}, Proposition \ref{prop:veefic_factors_2_nil}, and Remark \ref{rem:2nil_quotient_free}) and then define a homomorphism to a central kernel associated to a sort of ``difference" between two pinch maps (Lemma  \ref{lem:Deltac12_to_K}). We show this difference composed with $\pA(\vee_i f_i)$ is controlled by the ranks of the $\A^1$-degrees of the $f_i$ (Lemma \ref{overlineveefi_on_K}). Call these ranks the integer valued degrees of the $f_i$. This is used to show that a single example where Theorem \ref{thm:D-sum-all-f_i} holds for an $n$-tuple of endomorphisms with a given $n$-tuple of integer valued degrees implies the theorem for all endomorphisms of $\P^1$ with those integer valued degrees (Lemma~\ref{lem:key lemma}). This establishes the theorem for all $n$-tuples of endomorphisms of $\P^1$ with positive integer valued degrees (Corollary \ref{cor:d=dalg for positive ranks}). We then show a certain differences of differences is independent of one of the integer valued degrees (Lemma \ref{lem:f_i0+1Delta-fDelta}). This allows a downward induction which proves Theorem~\ref{thm:D-sum-all-f_i}.

As previously mentioned, the first step of our proof of Theorem~\ref{thm:D-sum-all-f_i} involves defining the \emph{unstable local $\mb{A}^1$-degree}.

\begin{defn}
    Let $f:\mb{P}^1_k\to\mb{P}^1_k$ be a pointed rational map. If $x$ is a closed point such that $f(x)=0$, then the \emph{unstable local degree} of $f$ at $x$ is the image $\deg^u_x(f)\in\GW^u(k)$ of the map
    \[\mb{P}^1_k\xrightarrow{c_x}\frac{\mb{P}^1_k}{\mb{P}^1_k-\{x\}}\xrightarrow{\overline{f}_x}\frac{\mb{P}^1_k}{\mb{P}^1_k-\{0\}}\xrightarrow{c_0^{-1}}\mb{P}^1_k.\]
\end{defn}

Here, $c_x$ denotes the collapse map, which is a weak equivalence when $x$ is $k$-rational. The last equivalence is given by inverse of the weak equivalence $c_0$. The map $\overline{f}_x$ is induced by $f$. Theorem~\ref{thm:D-sum-local} can be thought of as a Poincar\'e--Hopf theorem relating the global unstable degree to its local counterparts. We also give an explicit formula for the unstable local degree at rational points in terms of a ``higher residue.''

\begin{thm}
    Let $f/g$ be a pointed rational function. Let $r\in\mb{A}^1_k(k)$ be a root of $f$ of multiplicity $m$. Then there exists $a\in k^\times$ such that
    \[\frac{g(x)}{f(x)}=\frac{a}{(x-r)^m}+\sum_{i>-m}a_i(x-r)^i,\]
    and we have
    \[\deg^u_r(f)=\begin{cases}
        (\frac{m}{2}\mb{H},a^m) & m\equiv 0\mod 4,\\
        (\langle a\rangle+\frac{m-1}{2}\mb{H},a^m) & m\equiv 1\mod 4,\\
        (\frac{m}{2}\mb{H},-a^m) & m\equiv 2\mod 4,\\
        (\langle a\rangle+\frac{m-1}{2}\mb{H},-a^m) & m\equiv 3\mod 4.
    \end{cases}\]
    In other words, we have $\deg^u_r(f)=(A,\det{A})$, where we conflate the matrix
    \[A=\underbrace{\begin{pmatrix}
        * & * & \cdots & * & a\\
        * & * & \cdots & a & 0\\
        \vdots & \vdots & \text{\reflectbox{$\ddots$}} & \vdots & \vdots \\
        * & a & \cdots & 0 & 0\\
        a & 0 & \cdots & 0 & 0
    \end{pmatrix}}_{m\times m}\]
    with the isomorphism class of its associated bilinear form.
\end{thm}

\subsection{Outline}
We review some relevant terminology and notation in Section~\ref{sec:notation}. In Section~\ref{sec:local degree}, we define the unstable local $\mb{A}^1$-degree and derive an algebraic formula for it under nice hypotheses. In Section~\ref{sec:ltg homotopically}, we define the $D$-sum $\sum_D$ and prove that the unstable $\mb{A}^1$-degree satisfies a local-to-global principle with respect to $\sum_D$.

We take a slight detour in Section~\ref{sec:duplicants}, where we define a generalization of the polynomial discriminant (which we call the \emph{duplicant}). Code supporting our analysis of duplicants can be found in Appendix~\ref{sec:code}. Our aside on duplicants is utilized in Section~\ref{sec:algebraic local-to-global}, where we prove Theorem~\ref{thm:D-sum-local} (as Proposition~\ref{prop:local-to-global-rational-function}). Most of the techniques for this proof boil down to (somewhat involved) linear algebra. 

In Section~\ref{sec:general case}, we prove Theorem~\ref{thm:D-sum-all-f_i}, by inductively showing induced maps on $\pA$ are equal with Theorem~\ref{thm:D-sum-local} as the base case.

\subsection*{Acknowledgments}
We heartily thank Fabien Morel for useful discussions. We are also grateful to the anonymous referee for their careful reading and thorough feedback. JI, SS, and DT received support from an REU supplement to NSF DMS-2001890 and an REU supplement to NSF DMS-2103838. SM received support from NSF DMS-2202825. KW received support from NSF DMS Career DMS-2001890 (original number Career 1552730), NSF DMS-2103838 and NSF DMS-2405191.

\section{Terminology and notation}\label{sec:notation}
We will frequently work with \emph{pointed rational maps}, which are rational functions $f:\mb{P}^1_k\to\mb{P}^1_k$ such that $f(\infty)=\infty$. We will denote the unstable motivic homotopy category of pointed spaces over a field $k$ by $\mc{H}_\bullet(k)$. Given two pointed motivic spaces $X$ and $Y$, we denote the set of pointed $\mb{A}^1$-homotopy classes of maps $X\to Y$ by $[X,Y]$. We will really only need to consider the case of $X=Y=\mb{P}^1_k$.

\subsection{Unstable Grothendieck--Witt groups}\label{sec:unstable gw}
Define the unstable Grothendieck--Witt group
\[\GW^u(k):=\GW(k)\times_{k^\times/(k^\times)^2}k^\times.\]
We refer to the $\GW(k)$ and $k^\times$ factors of $\GW^u(k)$ as the \textit{stable} and \textit{unstable parts}, respectively. The group structure on $\GW^u(k)$ is given by $(\beta_1,d_1)+(\beta_2,d_2)=(\beta_1+\beta_2,d_1d_2)$ (or in words, by taking direct sums of the stable parts and multiplying the unstable parts). 

Equivalently, the unstable Grothendieck--Witt group $\GW^u(k)$ is the group completion of isomorphism classes of pairs $(\beta, (b_1,\ldots,b_n))$ where $\beta$ is a nondegenerate, symmetric bilinear form on a $k$-vector space with basis $(b_1,\ldots,b_n)$, and where an isomorphism is a linear isomorphism preserving the inner product and with determinant one in the given basis. This alternative definition is due to Lannes \cite[Remark~7.37]{Mor12}, but we could not find a proof in the literature and so provide one ourselves.

\begin{prop}
   Let $M(k)$ denote the monoid of isomorphism classes of pairs
   \[(\beta,(b_1,\ldots,b_n))\]
   as described above. Let $G(k)$ denote the group completion of $M(k)$. Then there is an isomorphism $G(k)\cong\GW^u(k)$.
\end{prop}
\begin{proof}
    By our first definition of the unstable Grothendieck--Witt group, $\GW^u(k)$ is the pullback (in abelian groups) of the cospan
    \[\begin{tikzcd}
        & k^\times\arrow[d,"\mr{mod}\ 2"]\\
        \GW(k)\arrow[r,"\disc"] & k^\times/(k^\times)^2,
    \end{tikzcd}\]
    where $\mr{mod}\ 2:k^\times\to k^\times/(k^\times)^2$ means modulo squares. By the universal property of pullbacks, it therefore suffices to construct homomorphisms $\phi:G(k)\to k^\times$ and $\psi:G(k)\to\GW(k)$ such that $\disc\circ\psi=\mr{mod}\ 2\circ\phi$.

    Let $\phi$ be the map induced by the Gram determinant map on $M(k)$. Explicitly, 
    \[\phi(\beta,(b_1,\ldots,b_n))=\det{B(\beta,(b_1,\ldots,b_n))},\]
    where $B(\beta,(b_1,\ldots,b_n))$ denotes the Gram matrix of the bilinear form $\beta$ in the basis $(b_1,\ldots,b_n)$. Note that $\phi$ is a well-defined function, as any linear isomorphism preserving the inner product with determinant one will preserve the Gram determinant. The monoid structure on $M(k)$ is given by direct sum. That is, if $(\beta,(b_1,\ldots,b_n)),(\gamma,(c_1,\ldots,c_m))\in M(k)$, then $(\beta,(b_i))+(\gamma,(c_j))=(\beta\oplus\gamma,(b_i\oplus\bm{0}_m,\ldots,{0}_n\oplus c_j)$, where $\bm{0}_a$ denotes the 0 vector of dimension $a$. Thus $B((\beta,(b_i))+(\gamma,(c_j)))=B(\beta,(b_i))\oplus B(\gamma,(c_j))$, so $\phi$ preserves monoid structures because determinants are multiplicative over direct sums of matrices. Group completing $M(k)$ gives us that $\phi$ is a homomorphism.

    Let $\psi$ denote the forgetful map. Explicitly, $\psi(\beta,(b_1,\ldots,b_n))=\beta$ on $M(k)$. Clearly, $\psi$ is a well-defined function that preserves units (sending the trivial bilinear form on the unique $0$-dimensional $k$-vector space with empty basis to $0\in\GW(k)$). To see that $\psi$ is a homomorphism, we simply observe that a direct sum of bilinear forms up to determinant one linear isomorphisms preserving the inner product will again be a direct sum up to linear isomorphism.

    Finally, $\mr{mod}\ 2\circ\phi$ is the Gram determinant modulo squares, which is precisely the definition of the discriminant of non-degenerate, symmetric bilinear forms. Thus $\mr{mod}\ 2\circ\phi=\disc\circ\psi$, as desired.
\end{proof}

We wish to describe $\GW^u(k)$ in terms of generators and relations. To this end, we recall the usual presentation of $\GW(k)$ \cite[9.4.~Corollary]{Sch85}.

\begin{prop}\label{prop:presentation of GW}
Let $k$ be a field. Given $a\in k^\times$, let $\langle a\rangle$ be the isomorphism class of the bilinear form $(x,y)\mapsto axy$. As a group, $\GW(k)$ is isomorphic to the abelian group generated by $\{\langle a\rangle:a\in k^\times\}$ modulo the following relations:
\begin{enumerate}[(i)]
\item $\langle ab^2\rangle=\langle a\rangle$ for all $a,b\in k^\times$.
\item $\langle a\rangle+\langle b\rangle=\langle a+b\rangle+\langle ab(a+b)\rangle$ for all $a,b\in k^\times$ such that $a+b\neq 0$.
\end{enumerate}
\end{prop}

\begin{rem}
Relations (i) and (ii) imply that
\[\langle a\rangle+\langle-a\rangle=\langle 1\rangle+\langle -1\rangle\]
for all $a\in k^\times$ (see e.g.~\cite[p.~30]{DGGM23}). Because of its ubiquity, we define the \textit{hyperbolic form} $\mb{H}:=\langle 1\rangle+\langle -1\rangle$.
\end{rem}

\begin{rem}
The Grothendieck--Witt group $\GW(k)$ can be given a ring structure by imposing the relation $\langle a\rangle\langle b\rangle=\langle ab\rangle$, corresponding to the tensor product of bilinear forms, but we will not need the ring structure in this article.
\end{rem}

Following the stable case, we can give a presentation of the unstable Grothendieck--Witt group in terms of generators and relations.

\begin{prop}\label{prop:presentation of GWu}
Let $k$ be a field. Given $a\in k^\times$, let $\langle a\rangle^u:=(\langle a\rangle,a)\in\GW^u(k)$. As a group, $\GW^u(k)$ is isomorphic to the abelian group $G(k)$ generated by $\{\langle a\rangle^u:a\in k^\times\}$ modulo the following relations:
\begin{enumerate}[(i')]
\item $\langle ab^2\rangle^u=\langle a\rangle^u+\langle b\rangle^u-\langle 1/b\rangle^u$ for all $a,b\in k^\times$.
\item $\langle a\rangle^u+\langle b\rangle^u=\langle 1/(a+b)\rangle^u+\langle ab(a+b)\rangle^u$ for all $a,b\in k^\times$ such that $a+b\neq 0$.
\end{enumerate}
\end{prop}
\begin{proof}
By definition, each element of $\GW^u(k)$ is of the form $(\beta,d)$, where $\beta\in\GW(k)$ and $d\in k^\times$ such that $d\equiv \disc\beta\mod(k^\times)^2$. Writing $\beta=\sum_{i=1}^n\langle a_i\rangle-\sum_{j=1}^m\langle b_j\rangle$ in $\GW(k)$, we have $d=c^2(\prod_i a_i)(\prod_j b_j^{-1})$ for some $c\in k^\times$. Since
\begin{align*}
    \langle c\rangle^u-\langle 1/c\rangle^u&=(\langle c\rangle,c)-(\langle 1/c\rangle,1/c)\\
    &=(\langle c\rangle,c)-(\langle c\rangle,1/c)\\
    &=(0,c^2)
\end{align*}
by Proposition~\ref{prop:presentation of GW} (i), we have $(\beta,d)=\langle c\rangle^u-\langle 1/c\rangle^u+\sum_{i=1}^n\langle a_i\rangle^u-\sum_{j=1}^m\langle b_j\rangle^u$. That is, $\GW^u(k)$ is generated by elements of the form $\langle a\rangle^u$. 

There are no relations on $\GW^u(k)$ imposed by the unstable factor $k^\times$, so we only need the additive relations on the stable factor given in Proposition~\ref{prop:presentation of GW}. Relation (i') is precisely Proposition~\ref{prop:presentation of GW} (i) when restricted to elements of the form $\langle a\rangle^u$. For relation (ii'), we have $\langle 1/(a+b)\rangle=\langle a+b\rangle$ in $\GW(k)$. To verify that the unstable factors agree, we simply compute $ab=\frac{1}{a+b}\cdot ab(a+b)$.

So far, we have proved that $\GW^u(k)$ is generated by the symbols $\langle a\rangle^u$, and that the relations (i') and (ii') hold in $\GW^u(k)$. This implies that the function $\vphi:G(k)\to\GW^u(k)$, defined on generators by
\[\vphi(\langle a\rangle^u)=(\langle a\rangle,a),\]
is a surjective homomorphism. Moreover, $\vphi$ is injective. Indeed, if $\vphi(\langle a\rangle^u)=\vphi(\langle b\rangle^u)$, then $(\langle a\rangle,a)=(\langle b\rangle,b)$ in $\GW^u(k)$, and hence $a=b$ as elements of $k^\times$. This implies that $\langle a\rangle^u=\langle b\rangle^u$. We conclude that $\vphi$ is an isomorphism, as desired.
\end{proof}

\begin{rem}
It is generally not true that $\langle a\rangle^u+\langle -a\rangle^u=\langle 1\rangle^u+\langle -1\rangle^u$, as the unstable parts of these elements are given by $-a^2$ and $-1$, respectively. In fact, using relation (ii'), we find that $\langle -a\rangle^u+\langle a-1\rangle^u=\langle -1\rangle^u+\langle a^2-a\rangle^u$, and hence
\begin{align*}
    \langle a\rangle^u+\langle -a\rangle^u&=\langle a\rangle^u+\langle a^2-a\rangle^u+\langle -1\rangle^u-\langle a-1\rangle^u\\
    &=\langle\tfrac{1}{a^2}\rangle^u+\langle a^4(a-1)\rangle^u+\langle-1\rangle^u-\langle a-1\rangle^u\tag{ii'}\\
    &=\langle\tfrac{1}{a^2}\rangle^u+\langle a-1\rangle^u+\langle a^2\rangle^u-\langle\tfrac{1}{a^2}\rangle^u+\langle -1\rangle^u-\langle a-1\rangle^u\tag{i'}\\
    &=\langle 1\rangle^u+\langle -1\rangle^u+\langle a\rangle^u-\langle\tfrac{1}{a}\rangle^u\tag{i'}\\
    &=\langle 1\rangle^u+\langle -1\rangle^u+(0,a^2).
\end{align*}
\end{rem}

\subsection{B\'ezoutians}
We will briefly recall some details about univariate B\'ezoutians, which provide an algebraic formula for the unstable degree by \cite{Caz12}.

\begin{defn}\label{def:bezoutian}
Given a pointed rational function\footnote{We will generally normalize $f/g$ so that $f$ is monic, but this assumption is not necessary to define the B\'ezoutian.} $f/g:\mb{P}^1_k\to\mb{P}^1_k$, the \textit{B\'ezoutian polynomial} of $f/g$ is defined to be
\[\Bez(f/g):=\frac{f(X)g(Y)-f(Y)g(X)}{X-Y}\in k[X,Y].\]

The \textit{B\'ezoutian matrix with respect to the monomial basis} is the matrix
\[\Bez^\mon(f/g):=(a_{ij})_{i,j=0}^m,\]
where $a_{ij}\in k$ are such that $\Bez(f/g)=\sum_{i,j}a_{ij}X^iY^j$.
\end{defn}

\begin{rem}
The term \textit{monomial basis} in Definition~\ref{def:bezoutian} refers to the monomial basis $\{x^i\}_i$ of the $k$-algebra $Q(f/g):=k[x,\frac{1}{g}]/(\frac{f}{g})$. The B\'ezoutian can be viewed as an element of $Q(f/g)\otimes_k Q(f/g)$ under the isomorphism
\begin{align*}
    Q(f/g)\otimes_k Q(f/g)&\to k[X,Y,\sfrac{1}{g(X)},\sfrac{1}{g(Y)}]/(\sfrac{f(X)}{g(X)},\sfrac{f(Y)}{g(Y)})\\
    a(x)\otimes b(x)&\mapsto a(X)b(Y).
\end{align*}
The B\'ezoutian matrix with respect to the monomial basis is then the coefficient matrix of the B\'ezoutian polynomial in the basis $\{X^iY^j\}_{i,j}$.

We will also need another choice of basis for $Q(f/g)$.
\end{rem}

\begin{defn}\label{def:newton basis}
Let $f/g:\mb{P}^1_k\to\mb{P}^1_k$ be a pointed rational function with rational root $r$ of order $m$. Consider the $k$-algebra
\[Q_r(f/g):=\frac{k[x,\sfrac{1}{g}]_{(x-r)}}{(f/g)}.\]
The \textit{local Newton basis} of $Q_r(f/g)$ is the basis
\[B^\Nwt_r(f/g):=\bigg\{\frac{f}{g\cdot (x-r)},\frac{f}{g\cdot (x-r)^2},\ldots,\frac{f}{g\cdot (x-r)^m}\bigg\}.\]
If all roots of $f$ are $k$-rational, then we define the \textit{(global) Newton basis} of $Q(f/g)$ as
\[B^\Nwt(f/g):=\bigcup_{r\in f^{-1}(0)}B^\Nwt_r(f/g).\]
\end{defn}

\begin{rem}
Any symmetric non-degenerate matrix $M$ over a field $k$ represents a symmetric non-degenerate bilinear form over $k$. Given such a matrix $M$, we will also denote the isomorphism class of the bilinear form that it represents by $M\in\GW(k)$.
\end{rem}

Cazanave computes the unstable global degree in terms of the B\'ezoutian with respect to the monomial basis~\cite[Theorem 3.6]{Caz12}.

\begin{thm}[Cazanave]\label{thm:cazanave}
There is a group isomorphism
\[\deg^u:([\mb{P}^1_k,\mb{P}^1_k],\nplus)^{\mr{gp}}\to\GW^u(k)\]
induced by the formula $\deg^u(f/g)= (\Bez^\mon(f/g),\det\Bez^\mon(f/g))$ for all pointed rational functions with $f$ monic.
\end{thm}

Here, the superscript $\mr{gp}$ denotes group completion (which is necessary, as the B\'ezoutian bilinear form only realizes elements of non-negative rank). The symbol $\nplus$ is Cazanave's \textit{na\"ive sum}, which is a monoid structure on the set $[\mb{P}^1_k,\mb{P}^1_k]$. We will recall the definition of $\nplus$ in Definition~\ref{def:n sum} when it becomes more relevant for us.

\begin{rem}
Note that $\Bez(cf/cg)=c^2\Bez(f/g)$. This $c^2$ factor does not cause any inconsistencies in the stable setting, as $\langle c^2\rangle=\langle 1\rangle$ in $\GW(k)$. However, this $c^2$ factor would cause $(\Bez^\mon(f/g),\det\Bez^\mon(f/g))$ to be ill-defined in $\GW^u(k)$. To get a well-defined B\'ezoutian, we therefore always normalize $f/g$ so that $f$ is monic as in \cite{Caz12}.
\end{rem}

When $f$ is a polynomial morphism, $\deg^u(f)$ is fully determined by the leading coefficient. Our convention that $f$ is monic forces $\deg^u(f)$ to scale inversely rather than directly:

\begin{prop}\label{prop:bezoutian of polynomial}
Let $f(x)=\sum_{i=0}^n a_ix^i\in k[x]$. Then $\deg^u(f)\in\GW^u(k)$ is presented by any matrix of the form
\begin{equation}\label{eq:upper triangular}
\begin{pmatrix}
* & * & \cdots & * & a_n^{-1}\\
* & * & \cdots & a_n^{-1} & 0\\
\vdots & \vdots & \text{\reflectbox{$\ddots$}} & \vdots & \vdots\\
* & a_n^{-1} & \cdots & 0 & 0\\
a_n^{-1} & 0 & \cdots & 0 & 0
\end{pmatrix}\qquad\text{or}\qquad
\begin{pmatrix}
0 & 0 & \cdots & 0 & a_n^{-1}\\
0 & 0 & \cdots & a_n^{-1} & *\\
\vdots & \vdots & \text{\reflectbox{$\ddots$}} & \vdots & \vdots\\
0 & a_n^{-1} & \cdots & * & *\\
a_n^{-1} & * & \cdots & * & *
\end{pmatrix}.
\end{equation}
\end{prop}
\begin{proof}
Because we normalize so that $f$ is monic, we write $f=\frac{x^n+\sum_i a_ia_n^{-1}x^i}{a_n^{-1}}$. One can readily compute that $\Bez(\frac{x^n+\sum_i a_ia_n^{-1}x^i}{a_n^{-1}})=a_n^{-1}\sum_{i+j=n-1}X^iY^j+\sum_{\ell=1}^{n-1} a_\ell a_n^{-1}\sum_{i+j=\ell-1}X^iY^j$, so the B\'ezoutian matrix with respect to the monomial basis is given by
\[\Bez^\mon(f)=\begin{pmatrix}
a_1a_n^{-1} & a_2a_n^{-1} & \cdots & a_{n-1}a_n^{-1} & a_n^{-1}\\
a_2a_n^{-1} & a_3a_n^{-1} & \cdots & a_n^{-1} & 0\\
\vdots & \vdots & \text{\reflectbox{$\ddots$}} & \vdots & \vdots\\
a_{n-1}a_n^{-1} & a_n^{-1} & \cdots & 0 & 0\\
a_n^{-1} & 0 & \cdots & 0 & 0
\end{pmatrix}.\]
The element of $\GW(k)$ determined by the matrix $\Bez^\mon(f)$ depends only on $a_n^{-1}$ (by e.g.~\cite[Lemma 6]{KW20}). Moreover, the determinant of any (anti)-triangular matrix is determined by its diagonal, so any matrix of the form Equation~\ref{eq:upper triangular} determines the same class in $\GW^u(k)$ as $(\Bez^\mon(f),\det\Bez^\mon(f))$.
\end{proof}

\section{Unstable local degree}\label{sec:local degree}
Following the stable setting, we will define the \textit{unstable local degree} of a map of curves at a closed point with rational image.

\begin{setup}\label{setup:local degree}
Let $X$ and $Y$ be curves over $k$. Let $f:X\to Y$ be a morphism. Assume that $x\in X$ is a closed point such that $f(x)\in Y(k)$. Let $U\subseteq X$ and $V\subseteq Y$ be Zariski open neighborhoods of $x$ and $f(x)$, respectively. Assume that $x$ is isolated in its fiber, so that (shrinking $U$ and $V$ as necessary) $f$ defines a map
\[\bar{f}_x:U/(U-\{x\})\to V/(V-\{f(x)\}).\]
By excision, we can rewrite this as
\[\bar{f}_x:\mb{P}^1_k/(\mb{P}^1_k-\{x\})\to\mb{P}^1_k/(\mb{P}^1_k-\{f(x)\})\simeq\mb{P}^1_k.\]
In order to obtain an element of $[\mb{P}^1_k,\mb{P}^1_k]$, we precompose with the collapse map $c_x:\mb{P}^1_k\to\mb{P}^1_k/(\mb{P}^1_k-\{x\})$.
\end{setup}

\begin{rem}\label{rem:collapse and induced map}
Suppose that $f$ has vanishing locus $D=\{x_1,\ldots,x_n\}$. We can then form the collapse map
\[c_D:\mb{P}^1_k\to\frac{\mb{P}^1_k}{\mb{P}^1_k-D}\]
from the inclusion $\mb{P}^1_k-D\hookrightarrow\mb{P}^1_k$. There is a canonical isomorphism $\mb{P}^1_k/(\mb{P}^1_k-D)\cong\bigvee_{i=1}^n\mb{P}^1_k/(\mb{P}^1_k-\{x_i\})$ \cite[Lemma~A.3]{Caz12}. The induced maps $\bar{f}_{x_i}:\mb{P}^1_k/(\mb{P}^1_k-\{x_i\})\to\mb{P}^1_k$ are constructed such that the diagram
\[\begin{tikzcd}
    \mb{P}^1_k\arrow[d,"c_D"']\arrow[r,"f"] & \mb{P}^1_k\\
    \frac{\mb{P}^1_k}{\mb{P}^1_k-D}\arrow[r,"\cong"] & \bigvee_i\frac{\mb{P}^1_k}{\mb{P}^1_k-\{x_i\}}\arrow[u,"\vee_i\bar{f}_{x_i}"']
\end{tikzcd}\]
commutes.
\end{rem}

\begin{defn}\label{defn:unstable local degree}
Assume the notation of Setup~\ref{setup:local degree}. The \textit{unstable local degree} of $f$ at $x$ is the image $\deg^u_x(f)\in\GW^u(k)$ of the composite $\bar{f}_x\circ c_x$ under Cazanave's isomorphism (Equation~\ref{eq:cazanave}). We will call $\deg^u_x(f)\in\GW^u(k)$ the \textit{algebraic} unstable local degree, in contrast to the \textit{homotopical} unstable local degree $\bar{f}_x\circ c_x\in[\mb{P}^1_k,\mb{P}^1_k]$.
\end{defn}

Note that if $x$ is the only zero of $f$, then the unstable degree coincides with the unstable local degree.

\begin{prop}\label{prop:local=global}
Let $f:\mb{P}^1_k\to\mb{P}^1_k$ be a pointed rational map with $f^{-1}(0)=\{x\}$. Assume that $x\in\mb{A}^1_k(k)$. Then $\deg^u_x(f)=\deg^u(f)$.
\end{prop}
\begin{proof}
By definition of the unstable local degree, it suffices to show that the diagram 
\begin{equation}\label{eq:local at one point}
    \begin{tikzcd}
        \mb{P}^1_k\arrow[d,"f"]\arrow[r,"c_x"] & \mb{P}^1_k/(\mb{P}^1_k-\{x\})\arrow[d,"\bar{f}_x"]\\
        \mb{P}^1_k & \mb{P}^1/(\mb{P}^1-\{0\})\arrow[l,"\simeq"]
    \end{tikzcd}
\end{equation}
commutes in $\H(k)$. The commutativity of Diagram~\ref{eq:local at one point} is explained in Remark~\ref{rem:collapse and induced map} (setting $n=1$).
\end{proof}

\begin{rem}
Precomposition with the collapse map should be thought of as a transfer $c_x^*:\GW^u(k(x))\to\GW^u(k)$, where $k(x)$ is the residue field of $x$. When $x$ is $k$-rational, the collapse map is in fact a homotopy equivalence $\mb{P}^1_k\simeq\mb{P}^1_k/(\mb{P}^1_k-\{x\})$ of pointed motivic spaces. Throughout this article, we will assume that $x$ is $k$-rational. We will give an analysis of the \textit{unstable transfer} $c_x^*$ and the unstable local degree at non-rational points in future work.
\end{rem}

\subsection{Algebraic formula for the unstable local degree}\label{sec:local degree formula}
We now give two formulas for the unstable local degree at rational points. The first formula assumes that we are computing the unstable local degree at a simple zero, in which case the local degree is given by the inverse of the derivative. This is the unstable analog of \cite[Lemma 9]{KW19}.

\begin{rem}
We are working with pointed rational functions $f/g$, which means that $\infty\in\mb{P}^1_k$ is not a root of $f$. In other words, all roots of $f$ lie in $\mb{A}^1_k=\mb{P}^1_k-\{\infty\}$.
\end{rem}

\begin{prop}\label{prop:local degree at simple zero}
Let $f:\mb{P}^1_k\to\mb{P}^1_k$ be a pointed rational map. Assume that $x\in\mb{A}^1_k(k)$ is a simple $k$-rational zero of $f$. Then $\deg^u_x(f)=\langle f'|_x^{-1}\rangle^u$.  
\end{prop}
\begin{proof}
This is the unstable, $k$-rational version of \cite[Proposition 15]{KW19}. Because the proof in \textit{loc.~cit.}~makes use of the stable motivic homotopy category, we need to modify the proof to hold in $\H(k)$.

Because $x$ is a simple zero of $f$ (equivalently, $f$ is \'etale at $x$), the induced map of tangent spaces $df_x:T_x\mb{P}^1_k\to f^*T_{f(x)}\mb{P}^1_k$ is a monomorphism. Thus $df_x$ induces a map $\Th(df_x):\Th(T_x\mb{P}^1_k)\to\Th(f^*T_{f(x)}\mb{P}^1_k)$ of Thom spaces. Because $x$ and $f(x)$ are $k$-rational, we have isomorphisms $\Th(T_x\mb{P}^1_k)\cong\Th(\mc{O}_{\Spec{k}})\cong\Th(f^*T_{f(x)}\mb{P}^1_k)$ in $\H(k)$, which fit into the commutative diagram
\begin{equation}\label{eq:deriv 1}
\begin{tikzcd}
\Th(T_x\mb{P}^1_k)\arrow[r,"\Th(df_x)"]\arrow[d,"\cong"'] & \Th(f^*T_{f(x)}\mb{P}^1_k)\arrow[d,"\cong"]\\
\Th(\mc{O}_{\Spec{k}})\arrow[r,"f'|_x"] & \Th(\mc{O}_{\Spec{k}}).
\end{tikzcd}
\end{equation}
Here, $f'|_x$ refers to the linear map $z\mapsto f'|_x\cdot z$. Note that $f'|_x\in k^\times$ since $f$ is \'etale at $x$. The naturality of the purity isomorphism \cite[Lemma 2.1]{Voe03} yields a commutative diagram
\begin{equation}\label{eq:deriv 2}
\begin{tikzcd}
&\Th(T_xU)\arrow[d,"\cong"']\arrow[r,"\Th(df_x)"] & \Th(f^*T_{f(x)}V)\arrow[d,"\cong"]&\\
\mb{P}^1_k\arrow[r,"\simeq"] &\frac{U}{U-\{x\}}\arrow[r,"f|_U"] & \frac{V}{V-\{f(x)\}}\arrow[r,"\simeq"] &\mb{P}^1_k.
\end{tikzcd}
\end{equation}
By stacking Diagrams~\ref{eq:deriv 1} and~\ref{eq:deriv 2}, we find that $\deg^u_x(f)=\deg^u(z\mapsto f'|_x\cdot z)$. In other words, we have reduced computing $\deg^u_x(f)$ to computing the unstable \textit{global} degree of a pointed rational function. We may therefore apply \cite{Caz12} and compute $\deg^u(z\mapsto f'|_x\cdot z)=\langle f'|_x^{-1}\rangle^u$ (see Proposition~\ref{prop:bezoutian of polynomial}).
\end{proof}

Now we give a more general, algebraic formula for the unstable local degree at rational points. This formula, which is the unstable analog of \cite[Main Theorem]{KW19} and \cite[Theorem~1.2]{BMP21}, involves the \textit{local Newton matrix} \cite[Definition 7]{KW20}.

\begin{defn}\label{def:newton matrix}
Let $f/g$ be a pointed rational function. Let $r\in\mb{A}^1_k(k)$ be a root of $f$ of multiplicity $m$. Write a partial fraction decomposition
\[\frac{g(x)}{f(x)}=\frac{A_{r,m}}{(x-r)^m}+\frac{A_{r,m-1}}{(x-r)^{m-1}}+\cdots+\frac{A_{r,1}}{x-r}+\text{higher order terms}.\]
Define the \textit{local Newton matrix}
\[\Nwt_r(f/g):=\begin{pmatrix}
A_{r,1} & A_{r,2} & \cdots & A_{r,m-1} & A_{r,m}\\
A_{r,2} & A_{r,3} & \cdots & A_{r,m} & 0\\
\vdots & \vdots & \text{\reflectbox{$\ddots$}} & \vdots & \vdots\\
A_{r,m-1} & A_{r,m} & \cdots & 0 & 0\\
A_{r,m} & 0 & \cdots & 0 & 0
\end{pmatrix}.\]
The local Newton matrix, together with its determinant, represents a class
\[\Nwt^u_r(f/g):=(\Nwt_r(f/g),\det\Nwt_r(f/g))\]
in $\GW^u(k)$.
\end{defn}

To prove that $\Nwt^u_r(f/g)$ computes $\deg^u_r(f/g)$, we first show that the unstable local degree is an $\mb{A}^1$-homotopy invariant (c.f.~\cite[Lemma 4]{KW20}).

\begin{lem}\label{lem:homotopy invariant}
Let $r\in\mb{A}^1_k$ be a closed point. Let $\frac{f_0}{g_0},\frac{f_1}{g_1}:\mb{P}^1_k\to\mb{P}^1_k$ be pointed rational functions such that $f_0(r)=f_1(r)=0$. Suppose there exists an open subscheme $U\subseteq\mb{A}^1_k\times\mb{A}^1_k$ containing $\{r\}\times\mb{A}^1_k$ and a morphism $H:U\to\mb{P}^1_k$ such that $H(x,0)=\frac{f_0}{g_0}(x)$ and $H(x,1)=\frac{f_1}{g_1}(x)$. If $\{r\}\times\mb{A}^1_k$ is a connected component of $H^{-1}(\{0\}\times\mb{A}^1_k)$, then
\[\deg^u_r(f_0/g_0)=\deg^u_r(f_1/g_1).\]
\end{lem}
\begin{proof}
    Let $Z$ be the union of the connected components of $H^{-1}(\{0\}\times\mb{A}^1_k)$ that are distinct from $\{r\}\times\mb{A}^1_k$. We can then write
    \begin{align*}
        \frac{U}{U-H^{-1}(0)}&=\frac{U}{U-((\{r\}\times\mb{A}^1_k)\amalg Z)}\\
        &\simeq\frac{U}{U-(\{r\}\times\mb{A}^1_k)}\vee\frac{U}{U-Z}\tag*{\cite[Lemma~A.3]{Caz12}}\\
        &\simeq\frac{\mb{P}^1_k\times\mb{A}^1_k}{\mb{P}^1_k\times\mb{A}^1_k-(\{r\}\times\mb{A}^1_k)}\vee\frac{U}{U-Z}.\tag{excision}
    \end{align*}
    This implies that the morphism $\frac{U}{U-H^{-1}(0)}\to\frac{\mb{P}^1_k}{\mb{P}^1_k-\{0\}}$ induced by $H$ is equivalent to a morphism
    \begin{equation}\label{eq:morphism}
        \frac{\mb{P}^1_k\times\mb{A}^1_k}{\mb{P}^1_k\times\mb{A}^1_k-(\{r\}\times\mb{A}^1_k)}\vee\frac{U}{U-Z}\to\frac{\mb{P}^1_k}{\mb{P}^1_k-\{0\}}.
    \end{equation}
    Pre-composing Equation~\ref{eq:morphism} with the natural morphisms
    \[\frac{\mb{P}^1_k}{\mb{P}^1_k-\{r\}}\times\mb{A}^1_k\to\frac{\mb{P}^1_k\times\mb{A}^1_k}{\mb{P}^1_k\times\mb{A}^1_k-(\{r\}\times\mb{A}^1_k)}\to\frac{\mb{P}^1_k\times\mb{A}^1_k}{\mb{P}^1_k\times\mb{A}^1_k-(\{r\}\times\mb{A}^1_k)}\vee\frac{U}{U-Z}\]
    gives us a na\"ive $\mb{A}^1$-homotopy from the map $\overline{(\frac{f_0}{g_0})}_r$ to $\overline{(\frac{f_1}{g_1})}_r$ (in the notation of Setup~\ref{setup:local degree}). It follows that we have a na\"ive homotopy from $\overline{(\frac{f_0}{g_0})}_r\circ c_r$ to $\overline{(\frac{f_1}{g_1})}_r\circ c_r$, and hence these maps determine the same element of $\GW^u(k)$.
\end{proof}

Using Lemma~\ref{lem:homotopy invariant}, we can now compute $\deg^u_r(f/g)=\Nwt^u_r(f/g)$ when $r$ is a rational point (c.f.~\cite[Corollary 8]{KW20}).

\begin{lem}\label{lem:degree=newton matrix}
Let $f/g$ be a pointed rational function. Let $r\in\mb{A}^1_k(k)$ be a root of $f$. Then
\[\deg^u_r(f/g)=\Nwt^u_r(f/g).\]
\end{lem}
\begin{proof}
Since $r$ is a root of $f$ of order $m$, there exist $A\in k^\times$ and a polynomial $f_0(x)\in k[x]$ such that $f(x)=(x-r)^m(A+(x-r)f_0(x))$. Similarly, since $f/g$ is a pointed rational function, $r$ is not a root of $g$ and hence there exist $B\in k^\times$ and a polynomial $g_0(x)\in k[x]$ such that $g(x)=B+(x-r)g_0(x)$.

Now let $U=\{(x,t)\in\mb{P}^1_k\times\mb{A}^1_k:x\neq\infty\text{ and }g(x)\neq 0\}$. Then
\[H_1(x,t)=\frac{(x-r)^m(A+t(x-r)f_0(x))}{g(x)}\]
determines a morphism $H_1:U\to\mb{P}^1_k$ such that $H_1(x,0)=\frac{A(x-r)^m}{g(x)}$ and $H_1(x,1)=\frac{f}{g}(x)$. This morphism satisfies the criteria of Lemma~\ref{lem:homotopy invariant}, which implies
\[\deg^u_r(f/g)=\deg^u_r(A(x-r)^m/g(x)).\]
Next, we get a morphism $H_2:\mb{P}^1_k\times\mb{A}^1_k\to\mb{P}^1_k$ given by
\[H_2(x,t)=\frac{A(x-r)^m}{B+t(x-r)g_0(x)}\]
that also satisfies the criteria of Lemma~\ref{lem:homotopy invariant}. Thus
\[\deg^u_r(A(x-r)^m/g(x))=\deg^u_r(A(x-r)^m/B).\]
Since $r$ is the only root of $A(x-r)^m/B$, it follows from Proposition~\ref{prop:local=global} that $\deg^u_r(f/g)=\deg^u(A(x-r)^m/B)$. We now normalize $A(x-r)^m/B=\frac{(x-r)^m}{B/A}$ and apply Proposition~\ref{prop:bezoutian of polynomial} to compute
\[\deg^u(\tfrac{(x-r)^m}{B/A})=\left(\begin{pmatrix}
* & * & \cdots & * & \tfrac{B}{A}\\
* & * & \cdots & \tfrac{B}{A} & 0\\
\vdots & \vdots & \text{\reflectbox{$\ddots$}} & \vdots & \vdots\\
* & \tfrac{B}{A} & \cdots & 0 & 0\\
\tfrac{B}{A} & 0 & \cdots & 0 & 0
\end{pmatrix},(-1)^{\frac{m(m-1)}{2}}\left(\frac{B}{A}\right)^m\right).\]
It thus suffices to prove that $\frac{B}{A}=A_{r,m}$. Given a rational function $F$, let $\Res^m(F,r)$ denote the coefficient of $(x-r)^{-m}$ in the Laurent expansion of $F$ about $r$,\footnote{One might call $\Res^m$ a \textit{higher residue}, since $\Res^1$ is the usual residue from complex analysis.} so that $A_{r,m}=\Res^m(g/f,r)$. Since $f(x)=A(x-r)^m(1+(x-r)f_0(x))$, we have
\[\frac{1}{f}=\frac{1}{A(x-r)^m}\sum_{i\geq 0}a_i(x-r)^i\]
with $a_0\in k^\times$ and $a_i\in k$ for $i>0$. Thus
\begin{align*}
    A_{r,m}&=\Res^m\big(\frac{g}{f},r\big)\\
    &=\Res^m\big(\frac{B+(x-r)g_0}{A(x-r)^m}\sum_{i\geq 0}a_i(x-r)^i,r\big)\\
    &=\frac{B}{A},
\end{align*}
as desired.
\end{proof}

\begin{rem}
    Lemma~\ref{lem:degree=newton matrix} corroborates Proposition~\ref{prop:local degree at simple zero}. If $f/g$ has a simple root at $r$, then Lemma~\ref{lem:degree=newton matrix} (in particular, its proof) implies that $\deg^u_r(f/g)=\langle\Res(g/f,r)\rangle^u$. The standard trick for computing the residue of a simple pole tells us
    \begin{align*}
        \Res(g/f,r)&=\frac{g(r)}{f'(r)}\\
        &=\frac{g(r)^2}{f'(r)\cdot g(r)-f(r)\cdot g'(r)}\\
        &=(f/g)'(r)^{-1},
    \end{align*}
    since $f(r)=0$. Thus $\langle\Res(g/f,r)\rangle^u=\langle(f/g)'|_r^{-1}\rangle^u$.
\end{rem}

\section{Local-to-global principle, homotopically}\label{sec:ltg homotopically}
Given a map $f:\mb{P}^1_k\to\mb{P}^1_k$, we are interested in understanding the relationship between the unstable degree $\deg^u(f)$ and the unstable local degrees $\deg^u_x(f)$ for $x\in f^{-1}(0)$. In particular, we would like to prove a \textit{local-to-global principle} or \textit{local decomposition} for the homotopical unstable degree $\deg^u(f)$, namely that
\begin{align}\label{eq:local decomp def}
\deg^u(f)=\sum_{x\in f^{-1}(0)}\deg^u_x(f).
\end{align}
In topology, such local decompositions give rise to the Poincar\'e--Hopf theorem for vector bundles. A crucial aspect of Equation~\ref{eq:local decomp def} is that the sum is indexed over the vanishing locus $f^{-1}(0)$ --- we do not only want to express $\deg^u(f)$ in terms of simpler summands, but rather that these summands have an explicit and tractable geometric relationship to the morphism $f$.

In this section, we will prove a homotopical local decomposition 
\begin{align}\label{eq:homotopical local decomp}
f=\sum_{x\in f^{-1}(0)}\bar{f}_x\circ c_x.
\end{align}
In Section~\ref{sec:algebraic local-to-global}, we will obtain an algebraic local decomposition $\deg^u(f)=\sum_{x\in f^{-1}(0)}\deg^u_x(f)$ by analyzing the image of Equation~\ref{eq:homotopical local decomp} in $\GW^u(k)$. We will also discuss Cazanave's decomposition of $\deg^u(f)$ and how it fails to be local. 

Homotopically, sums of maps are given by pinching and folding. That is, given $f,g:X\to Y$, the sum $f+g$ is defined as the composite
\[X\xrightarrow{\cv}X\vee X\xrightarrow{f\vee g}Y\vee Y\xrightarrow{\nabla}Y.\]
The fold is actually unnecessary for our purposes: the wedge is the coproduct in pointed spaces, so maps out of the wedge are in bijection with a set of maps out of each to a fixed target. Post-composition with the fold map would be necessary if we were working with an external wedge sum, which we will not need in this article.

Whenever $X$ is a suspension $X\simeq S^1\wedge X'$, we can construct a pinch map as follows. Any choice of inclusion $S^0\subset S^1$ separates $S^1$ into two disjoint intervals; collapsing $S^0$ closes each of these intervals off into an $S^1$, with the two copies of $S^1$ joined together at the image of $S^0$ (see Figure~\ref{fig:pinch}). One then defines the pinch $\cv:X\to X\vee X$ as
\[S^1\wedge X'\xrightarrow{\cv}(S^1\vee S^1)\wedge X'\simeq(S^1\wedge X')\vee(S^1\wedge X').\]
Here, the last homotopy equivalence holds in any category where smash products distribute over wedge sums, i.e.~any category in which products commute with pushouts.

\begin{figure}
\centering
\begin{tikzpicture}
\draw[very thick] (0,0) circle (20pt);
\draw[very thick] (110pt,0) circle (15pt);
\draw[very thick] (140pt,0) circle (15pt);
\filldraw[red] (0,20pt) circle (2pt);
\filldraw[red] (0,-20pt) circle (2pt);
\filldraw[red] (125pt,0) circle (2pt);
\node[red] (0,0) {$S^0$};
\draw[->,thick] (40pt,0) -- (75pt,0);
\end{tikzpicture}
\caption{Pinching $S^1$}\label{fig:pinch}
\end{figure}
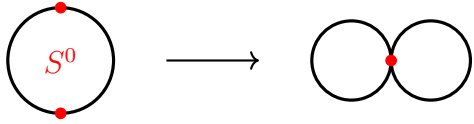

In order to add pointed endomorphisms of $\mb{P}^1$, we need a workable pinch map $\mb{P}^1\to\mb{P}^1\vee\mb{P}^1$. While $\mb{P}^1\simeq S^1\wedge\mb{G}_m$ as a motivic space, the simplicial pinch map $\mb{P}^1\to\mb{P}^1\vee\mb{P}^1$ is unwieldy from the perspective of algebraic geometry. That is, there is not an evident way to describe the simplicial pinch in terms of subschemes of $\mb{P}^1$. This stems from the fact that we need $\mb{A}^1$-invariance to realize $\mb{P}^1$ as a suspension:
\[\begin{tikzcd}
\mb{G}_m\arrow[r]\arrow[d] & *\arrow[d]\\
*\arrow[r] & S^1\wedge\mb{G}_m
\end{tikzcd}
\qquad\simeq\qquad
\begin{tikzcd}
\mb{G}_m\arrow[r]\arrow[d] & \mb{A}^1\arrow[d]\\
\mb{A}^1\arrow[r] & \mb{P}^1.
\end{tikzcd}\]
While the simplicial pinch map gives the usual group structure on $[\mb{P}^1_k,\mb{P}^1_k]\cong\GW^u(k)$ \cite[Lemma 3.20 and Theorem 3.21]{Caz12}, Cazanave noticed that the collapse map can be viewed as an algebraic pinch map \cite[Lemma A.3]{Caz12}. Cazanave used these algebraic pinch maps to define the na\"ive sum $\nplus:[\mb{P}^1,\mb{P}^1]^2\to[\mb{P}^1,\mb{P}^1]$ \cite[\S 3.1]{Caz12}, which give a method for decomposing global maps into ``local'' terms. However, as we will describe in Section~\ref{sec:naive insufficient}, the na\"ively local terms of a map $f:\mb{P}^1\to\mb{P}^1$ fail to be truly local.

While Cazanave only considers the pinch map arising from the collapse map $c_{\{0,\infty\}}:\mb{P}^1_k\to\mb{P}^1_k/(\mb{P}^1_k-\{0,\infty\})$, we will need to consider the pinch maps arising from $c_D:\mb{P}^1_k\to\mb{P}^1_k/(\mb{P}^1_k-D)$ for arbitrary divisors $D\subset\mb{P}^1_k(k)$. We begin by defining the algebraic pinch map associated to $D$.

\begin{lem}\label{lem:inverse to collapse}
Let $x\in\mb{P}^1_k(k)$ be a rational point. Then there exists a homotopy inverse $p:\frac{\mb{P}^1_k}{\mb{P}^1_k-\{x\}}\to\mb{P}^1_k$ to $c_x$ in $\H(k)$.
\end{lem}
\begin{proof}
The collapse map $c_x$ is a homotopy equivalence by \cite[Lemma 5.4]{Hoy14}, which implies the existence of a homotopy inverse $p$. In fact, an explicit formula for $p$ is given in \textit{loc.~cit}.
\end{proof}

\begin{defn}\label{def:D-pinch}
Let $D=\{x_1,\ldots,x_n\}\subset\mb{P}^1_k(k)$ be a finite set of rational points. Define the \textit{$D$-pinch map} as the composite
\[\cv_D:\mb{P}^1_k\xrightarrow{c_D}\frac{\mb{P}^1_k}{\mb{P}^1_k-D}\xrightarrow{\simeq}\bigvee_{i=1}^n\frac{\mb{P}^1_k}{\mb{P}^1_k-\{x_i\}}\xrightarrow{\vee_i p_i}\bigvee_{i=1}^n\mb{P}^1_k,\]
where $c_D$ is the collapse map induced by the inclusion $\mb{P}^1_k-D\hookrightarrow\mb{P}^1_k$, the second map is the canonical isomorphism of motivic spaces $\mb{P}^1_k/(\mb{P}^1_k-D)\cong\bigvee_{i=1}^n\mb{P}^1_k/(\mb{P}^1_k-\{x_i\})$ given by \cite[Lemma A.3]{Caz12}, and $p_i=c^{-1}_{x_i}$ (which exists by Lemma~\ref{lem:inverse to collapse}) for each $i$.
\end{defn}

Homotopically, the desired local-to-global principle for the unstable degree should be encoded as the commutativity of the following diagram, which relates our ``global'' map $f:\mb{P}^1_k\to\mb{P}^1_k$ to an appropriate sum $\vee_i(\bar{f}_{x_i}\circ c_{x_i})\circ\cv_D:\mb{P}^1_k\to\mb{P}^1_k$ of its local terms.
\begin{equation}\label{eq:homotopical local decomp diagram}
\begin{tikzcd}
\mb{P}^1_k\arrow[rr,"c_D"]\arrow[d,"f"'] && \frac{\mb{P}^1_k}{\mb{P}^1_k-D}\arrow[rr,"\cong"]\arrow[drr,"\cong"'] && \bigvee_i\frac{\mb{P}^1_k}{\mb{P}^1_k-\{x_i\}}\arrow[rr,"\vee_ip_i"]\arrow[d,equal] && \bigvee_i\mb{P}^1_k\arrow[d,equal]\\
\mb{P}^1_k &&  && \bigvee_i\frac{\mb{P}^1_k}{\mb{P}^1_k-\{x_i\}}\arrow[llll,"\vee_i\bar{f}_{x_i}"] && \bigvee_i\mb{P}^1_k\arrow[ll,"\vee_ic_{x_i}"]
\end{tikzcd}    
\end{equation}

\begin{thm}[Local-to-global principle, homotopically]\label{thm:local-global homotopic}
Let $f:\mb{P}^1_k\to\mb{P}^1_k$ be a pointed rational map with vanishing locus $D=\{x_1,\ldots,x_n\}\subset\mb{P}^1_k(k)$. Then $f=\vee_i(\bar{f}_{x_i}\circ c_{x_i})\circ\cv_D$ in $\H(k)$.
\end{thm}
\begin{proof}
The top three maps of Diagram~\ref{eq:homotopical local decomp diagram} compose to $\cv_D$. Thus if Diagram~\ref{eq:homotopical local decomp diagram} commutes in $\H(k)$, then we obtain the desired result by comparing the leftmost vertical map with the composite around the remaining three edges of the outer rectangle. 

There are three polygons in Diagram~\ref{eq:homotopical local decomp diagram} to consider. The commutativity of the central triangle
\[\begin{tikzcd}
\phantom{\circ}\arrow[rr,"\cong"]\arrow[drr,"\cong"'] && \phantom{\circ}\arrow[d,equal]\\
&& \phantom{\circ}
\end{tikzcd}\]
is simply two copies of the isomorphism $\frac{\mb{P}^1_k}{\mb{P}^1_k-D}\cong\bigvee_i\frac{\mb{P}^1_k}{\mb{P}^1_k-\{x_i\}}$ \cite[Lemma A.3]{Caz12}. The commutativity of the rightmost rectangle
\[\begin{tikzcd}
\phantom{\circ}\arrow[rr,"\vee_ip_i"]\arrow[d,equal] && \phantom{\circ}\arrow[d,equal]\\
\phantom{\circ} && \phantom{\circ}\arrow[ll,"\vee_i c_{x_i}"]
\end{tikzcd}\]
follows from Lemma~\ref{lem:inverse to collapse}, which states that $p_i$ is the homotopy inverse of $c_{x_i}$ in $\H(k)$.

Finally, we need to show that the leftmost trapezoid
\[\begin{tikzcd}
\phantom{\circ}\arrow[d,"f"']\arrow[rr,"c_D"] && \phantom{\circ}\arrow[drr,"\cong"] &&\\
\phantom{\circ} &&  && \phantom{\circ}\arrow[llll,"\vee_i\bar{f}_{x_i}"]
\end{tikzcd}\]
commutes. The commutativity of this diagram is explained in Remark~\ref{rem:collapse and induced map}.
\end{proof}

In summary, we have proved that a pointed rational function is homotopic to the sum of its homotopical local unstable degrees. The subtlety in this story is figuring out \textit{which definition} of addition ensures this local-to-global principle. Theorem~\ref{thm:local-global homotopic} states that taking our addition to be $(-)\circ\cv_D$, where $D$ is the vanishing locus of $f:\mb{P}^1_k\to\mb{P}^1_k$, gives us the desired local-to-global principle for $f$. This justifies the following definition.

\begin{defn}\label{def:D-sum}
Let $D=\{r_1,\ldots,r_n\}\subset\mb{A}^1_k(k)$. The \emph{(homotopical) $D$-sum} is the function
\[\sum_D:=(-)\circ\cv_D:[\mb{P}^1_k,\mb{P}^1_k]^n\to[\mb{P}^1_k,\mb{P}^1_k].\]
If we do not wish to specify the divisor $D$, we will refer to the $D$-sum as a \emph{(homotopical) divisorial sum}.
\end{defn}

Our next goal is to study the algebraic image $\bigoplus_D:=\deg^u\circ\sum_D$ of the $D$-sum and compare it to the usual group structure on $\GW^u(k)$.

\section{Aside on duplicants}\label{sec:duplicants}
Before computing the addition law $\bigoplus_D$ in $\GW^u(k)$, we need to generalize the notion of the discriminant of a polynomial. We begin with some notation.

\begin{notn}\label{notn:duplicant}
    Given $m,n\in\mb{N}$, denote the $m\textsuperscript{th}$ elementary symmetric polynomial in $n$ variables by
    \[\sigma_{m,n}(x_1,\ldots,x_n):=\sum_{1\leq i_1<\ldots<i_m\leq n}x_{i_1}\cdots x_{i_m}.\]
    By convention, we will set $\sigma_{0,n}=1$ and $\sigma_{m,n}=0$ for $m\not\in\{0,\ldots,n\}$. 
    
    Given a monic polynomial of the form $f=\prod_{i=1}^n(x-r_i)^{e_i}$, let $N:=\deg(f)$ and 
    \[\bm{r}_{i,j}:=(\underbrace{r_1,\ldots,r_1}_{e_1\text{ times}},\ldots,\underbrace{r_i,\ldots,r_i}_{e_i-j\text{ times}},\ldots,\underbrace{r_n,\ldots,r_n}_{e_n\text{ times}}).\]
    By Vieta's formulas, the coefficient of $x^i$ in $f/(x-r_\ell)^j=(x-r_\ell)^{e_\ell-j}\prod_{m\neq \ell}(x-r_m)^{e_m}$ is given by $(-1)^{N-i-j}\sigma_{N-i-j,N-j}(\bm{r}_{\ell,j})$. For fixed $\ell$ and varying $0\leq i\leq N-1$ and $1\leq j\leq e_\ell$, we get a matrix of coefficients $\Sigma_\ell(f):=((-1)^{N-i-j}\sigma_{N-i-j,N-j}(\bm{r}_{\ell,j}))_{i,j}$. If we treat $i$ as the row index and $j$ as the column index, then the matrix
    \[\Sigma(f):=\begin{pmatrix} \Sigma_1(f) & \Sigma_2(f) & \cdots & \Sigma_n(f)\end{pmatrix}\]
    is an $N\times N$ square. We will only be interested in $\det\Sigma(f)$ and its square, so we will conflate $\Sigma(f)$ and its transpose $\Sigma(f)^\intercal$ when convenient.
\end{notn}

The heavy notation needed for this setup is unfortunate, as it may obfuscate what $\Sigma(f)$ really is:

\begin{prop}\label{prop:change of basis}
    Let $f/g:\mb{P}^1_k\to\mb{P}^1_k$ be a pointed rational function. Assume that $f=\prod_{i=1}^n(x-r_i)^{e_i}$ with $N:=\sum_{i=1}^n e_i$. Then the change-of-basis matrix from the monomial basis
    \[\left\{\frac{1}{g(x)},\frac{x}{g(x)},\ldots,\frac{x^{N-1}}{g(x)}\right\}\]
    to the Newton basis
    \[\left\{\frac{f(x)}{(x-r_1)g(x)},\ldots,\frac{f(x)}{(x-r_1)^{e_1}g(x)},\ldots,\frac{f(x)}{(x-r_n)g(x)},\ldots,\frac{f(x)}{(x-r_n)^{e_n}g(x)}\right\}\]
    is given by $\Sigma(f)^\intercal$.
\end{prop}
\begin{proof}
    By definition, $\Sigma_\ell(f)$ is the matrix of coefficients of $f/(x-r_\ell),\ldots,f/(x-r_\ell)^{e_\ell}$. This matrix is indexed so that
    \[\Sigma_\ell(f)^\intercal\begin{pmatrix}\frac{1}{g(x)}\\ \vdots\\ \frac{x^{N-1}}{g(x)}\end{pmatrix}=\begin{pmatrix}\frac{f(x)}{(x-r_\ell)g(x)}\\ \vdots\\ \frac{f(x)}{(x-r_\ell)^{e_\ell}g(x)}\end{pmatrix}.\]
    It follows that $\Sigma(f)^\intercal$ is the desired change-of-basis matrix.
\end{proof}

\begin{rem}\label{rem:independent of g}
    Note that the change-of-basis matrix in Proposition~\ref{prop:change of basis} does not depend on $g(x)$, justifying the notation $\Sigma(f)$.
\end{rem}

We will need to work with $\det\Sigma(f)^2$ in Section~\ref{sec:algebraic local-to-global}, so we give it a name and derive a formula for it.

\begin{defn}\label{def:disc}
    Let $f\in k[x]$ be a monic polynomial whose roots are all $k$-rational. Under the conventions listed in Notation~\ref{notn:duplicant}, we define the \textit{duplicant} of $f$ as
    \[\D(f):=\det\Sigma(f)^2.\]
\end{defn}

\begin{ex}
    Let $f=(x-r_1)(x-r_2)^2$. Then $\Sigma_1(f)=\begin{pmatrix} r_2^2 & -2r_2 & 1\end{pmatrix}$ and
    \[\Sigma_2(f)=\begin{pmatrix}
        r_1r_2 & -r_1-r_2 & 1\\
        -r_1 & 1 & 0
    \end{pmatrix}.
    \]
    Setting $f_\red=(x-r_1)(x-r_2)$, we compute
    \begin{align*}
    \D(f)&=\det\begin{pmatrix}
    r_2^2 & -2r_2 & 1\\
    r_1r_2 & -r_1-r_2 & 1\\
    -r_1 & 1 & 0\end{pmatrix}^2\\
    &=(r_1-r_2)^4\\
    &=\disc(f_\red)^2.
    \end{align*}
    See Appendix~\ref{sec:code} for some rough Sage code for computing duplicants.
\end{ex}

The following proposition shows that the duplicant is indeed a generalization of the discriminant.

\begin{prop}\label{prop:duplicant=discriminant}
    Let $f=\prod_{i=1}^n(x-r_i)$ with all $r_i$ distinct. Then $\D(f)=\disc(f)$.
\end{prop}
\begin{proof}
    Since $e_i=1$ for all $i$, the matrices of coefficients take the form
    \[\Sigma_\ell(f):=((-1)^{N-i-1}\sigma_{N-i-1,N-1}(r_1,\ldots,\hat{r}_\ell,\ldots,r_n))_{i=0}^{N-1}.\]
    Note that
    \begin{align*}
        \frac{\partial\sigma_{a,b}}{\partial x_\ell}&=\sum_{1\leq i_1<\ldots<\ell<\ldots<i_a\leq b}x_{i_1}\cdots\hat{x}_\ell\cdots x_{i_a}\\
        &=\sigma_{a-1,b-1}(x_1,\ldots,\hat{x}_\ell,\ldots,x_b)
    \end{align*}
    when $1\leq a\leq b$. It follows that, up to multiplying some rows by $-1$, we have
    \begin{align*}
    \Sigma(f)&=\begin{pmatrix}
    \sigma_{n-1,n-1}(\bm{r}_{1,1}) & \sigma_{n-2,n-1}(\bm{r}_{1,1}) & \cdots & \sigma_{0,n-1}(\bm{r}_{1,1}) \\
    \sigma_{n-1,n-1}(\bm{r}_{2,1}) & \sigma_{n-2,n-1}(\bm{r}_{2,1}) & \cdots & \sigma_{0,n-1}(\bm{r}_{2,1}) \\
    \vdots & \vdots & \ddots & \vdots \\
    \sigma_{n-1,n-1}(\bm{r}_{n,1}) & \sigma_{n-2,n-1}(\bm{r}_{n,1}) & \cdots & \sigma_{0,n-1}(\bm{r}_{n,1})\end{pmatrix}\\
    &=\begin{pmatrix}
    \frac{\partial\sigma_{n,n}}{\partial x_1} & \frac{\partial\sigma_{n-1,n}}{\partial x_1} & \cdots & \frac{\partial\sigma_{1,n}}{\partial x_1}\\
    \frac{\partial\sigma_{n,n}}{\partial x_2} & \frac{\partial\sigma_{n-1,n}}{\partial x_2} & \cdots & \frac{\partial\sigma_{1,n}}{\partial x_2} \\
    \vdots & \vdots & \ddots & \vdots\\
    \frac{\partial\sigma_{n,n}}{\partial x_n} & \frac{\partial\sigma_{n-1,n}}{\partial x_n} & \cdots & \frac{\partial\sigma_{1,n}}{\partial x_n}
    \end{pmatrix}\Bigg|_{x_i=r_i},
    \end{align*}
    where the evaluation sets $x_i=r_i$ for all $1\leq i\leq n$. Thus 
    \[\det\Sigma(f)=\pm\Jac(\sigma_{n,n},\ldots,\sigma_{1,n})|_{x_i=r_i}.\]
    In order to compute $\det\Sigma(f)^2$, it therefore suffices to evaluate the Jacobian determinant of the elementary symmetric polynomials. The computation
    \[\Jac(\sigma_{1,n},\ldots,\sigma_{n,n})=\prod_{1\leq i<j\leq n}(x_i-x_j)\]
    is classical (see e.g.~\cite[pp.~150]{Per51}) and implies 
    \[\Jac(\sigma_{n,n},\ldots,\sigma_{1,n})=(-1)^{\lfloor n/2\rfloor}\prod_{1\leq i<j\leq n}(x_i-x_j).\]
    After evaluating $x_i\mapsto r_i$, this squares to $\disc(f)$.
\end{proof}

Based on computations using the code in Appendix~\ref{sec:code}, we can conjecture (and subsequently prove) a compact formula for $\D(f)$.

\begin{thm}\label{thm:duplicant}
    If $f=\prod_{i=1}^n(x-r_i)^{e_i}$, then
    \[\det\Sigma(f)=\pm\prod_{1\leq i<j\leq n}(r_i-r_j)^{e_ie_j},\]
    and hence
    \[\D(f)=\prod_{1\leq i<j\leq n}(r_i-r_j)^{2e_ie_j}.\]
\end{thm}
\begin{proof}
Let $N:=\sum_{i=1}^n e_i$. Consider the monomial, slant monomial, and Newton bases of $Q(f):=k[x]/(f)$:
\begin{align*}
    B^\mon(f)&=\bigg\{1,x,x^2,\ldots,x^{N-1}\bigg\},\\
    B^\slant(f)&=\bigg\{1,(x-r_1),\ldots,(x-r_1)^{e_1},\\
    &\qquad (x-r_1)^{e_1}(x-r_2),\ldots,(x-r_1)^{e_1}(x-r_2)^{e_2},\\
    &\qquad\ldots,\\
    &\qquad\prod_{i=1}^{n-1}(x-r_i)^{e_i}\cdot(x-r_n),\ldots,\prod_{i=1}^{n-1}(x-r_i)^{e_i}\cdot(x-r_n)^{e_n-1}\bigg\},\\
    B^\Nwt(f)&=\bigcup_{i=1}^n\bigg\{\frac{f}{x-r_i},\ldots,\frac{f}{(x-r_i)^{e_i}}\bigg\}.
\end{align*}
Given bases $B$ and $B'$, denote the $B$-to-$B'$ change-of-basis matrix by $T^B_{B'}$. To simplify notation, we will write $T^{\mon}_{\Nwt}(f):=T^{B^\mon(f)}_{B^\Nwt(f)}$, and similarly for other pairs of bases among $B^\mon(f),B^\slant(f),B^\Nwt(f)$. By Proposition~\ref{prop:change of basis}, we can prove the present theorem by showing that $\det{T^{\mon}_{\Nwt}(f)}=\pm\prod_{i<j}(r_i-r_j)^{e_ie_j}$. 

Note that $T^{\mon}_{\slant}(f)$ is a triangular matrix with all entries on the diagonal equal to 1, since the elements of $\mon(f)$ and $\slant(f)$ are monic polynomials of degrees $0,1,\ldots,N-1$. In particular, $\det{T^{\mon}_{\slant}}(f)=1$, so $\det{T^{\mon}_{\Nwt}(f)}=\det{T^{\slant}_{\Nwt}(f)}$. We will thus compute $\det{T^{\slant}_{\Nwt}(f)}$.

We conclude the proof by inducting on $n$. The base case is $n=1$, in which $T^\slant_\Nwt(f)$ is a permutation matrix (and thus has determinant $\pm 1$) and $\prod_{1\leq i<j\leq n}(r_i-r_j)^{e_ie_j}$ is an empty product (and thus equal to 1). As the inductive hypothesis, we may therefore assume
\[\det{T^\slant_\Nwt(\tilde{f})}=\pm\prod_{1\leq i<j\leq n-1}(r_i-r_j)^{e_ie_j},\]
where $\tilde{f}=\prod_{i=1}^{n-1}(x-r_i)^{e_i}$ (so that $f=\tilde{f}\cdot(x-r_n)^{e_n}$). We will complete the inductive step in Lemma~\ref{lem:inductive step}.
\end{proof}

\begin{lem}\label{lem:inductive step}
Assume the notation of Theorem~\ref{thm:duplicant} and its proof. If $\det{T^\slant_\Nwt(\tilde{f})}=\pm\prod_{1\leq i<j\leq n-1}(r_i-r_j)^{e_ie_j}$, then $\det{T^\slant_\Nwt(f)}=\pm\prod_{1\leq i<j\leq n}(r_i-r_j)^{e_ie_j}$.
\end{lem}
\begin{proof}
    Note that
\begin{align}\label{eq:slant bases}
    B^\Nwt(f)&=\bigg\{v(x)\cdot(x-r_n)^{e_n}:v(x)\in B^\Nwt(\tilde{f})\bigg\}\cup\bigg\{\frac{f}{x-r_n},\ldots,\frac{f}{(x-r_n)^{e_n}}\bigg\},\nonumber\\
    B^\slant(f)&=B^\slant(\tilde{f})\cup\bigg\{\frac{f}{(x-r_n)^{e_n}},\ldots,\frac{f}{x-r_n}\bigg\}.
\end{align}
    This implies that $T^\slant_\Nwt(f)$ is a block diagonal matrix: the rows of $T^\slant_\Nwt(f)$ corresponding to the $\{f/(x-r_n),\ldots,f/(x-r_n)^{e_n}\}$ are 0 in the columns corresponding to $B^\slant(\tilde{f})$ and a permutation matrix in the remaining columns. Similarly, the rows of $T^\slant_\Nwt(f)$ corresponding to the elements $\{v(x)\cdot(x-r_n)^{e_n}:v(x)\in B^\Nwt(\tilde{f})\}$ are the first $\sum_{i=1}^{n-1}e_i$ rows of the product
    \[M\cdot T^\slant_\Nwt(\tilde{f})\]
    (followed by $e_n$ columns of zeros), where $M$ is the $N\times(\sum_{i=1}^{n-1}e_i)$ matrix corresponding to the linear transformation $Q(\tilde{f})\to Q(f)$ given by multiplication by $(x-r_n)^{e_n}$ on $B^\slant(\tilde{f})$.

    By Equation~\ref{eq:slant bases}, the first $\sum_{i=1}^{n-1}e_i$ rows of $M$ correspond to the elements of $B^\slant(\tilde{f})$. In particular, the matrix $M$ consists of a square matrix $S$ with rows and columns indexed by $B^\slant(\tilde{f})$, followed by $e_n$ rows underneath that are irrelevant for our computations. The matrix $S$ can be written as $P\cdot M$, where $P$ is the matrix of the projection $Q(f)\to Q(\tilde{f})$ corresponding to forgetting the basis elements $B^\slant(f)-B^\slant(\tilde{f})=\{f/(x-r_n)^{e_n},\ldots,f/(x-r_n)\}$.

    All of this setup allows us to state
    \begin{align*}
    \det{T^\slant_\Nwt(f)}&=\pm\det\big(M\cdot T^\slant_\Nwt(\tilde{f})\big)_{i,j=1}^{N-e_n}\\
    &=\pm\det(P\cdot M)\cdot\det{T^\slant_\Nwt(\tilde{f})}.
    \end{align*}
    It thus suffices to prove that $\det(P\cdot M)=\prod_{i=1}^{n-1}(r_i-r_n)^{e_ie_n}$. Note that if we write
    \[F=\prod_{i=1}^n\prod_{j=1}^{e_i}(x-r_{i,j})\]
    and treat $r_{i,j}$ as variables, then $B^\slant(F)$ is a basis for the free $k[r_{1,1},\ldots,r_{n,e_n}]$-module given by polynomials in $k[r_{1,1},\ldots,r_{n,e_n}][x]$ of degree at most $N-1$. Similarly, writing
    \[\tilde{F}=\prod_{i=1}^{n-1}\prod_{j=1}^{e_i}(x-r_{i,j}),\]
    we have that $B^\slant(\tilde{F})$ is a basis for the free $k[r_{1,1},\ldots,r_{n-1,e_{n-1}}]$-module given by polynomials in $k[r_{1,1},\ldots,r_{n-1,e_{n-1}}][x]$ of degree at most $N-e_n-1$. Specializing $r_{i,j}\mapsto r_i$ sends $F\mapsto f$ and $\tilde{F}\mapsto\tilde{f}$. In particular, we can compute $\det(P\cdot M)$ by working with $B^\slant(F)$ and $B^\slant(\tilde{F})$ and then specializing. By inductively specializing, beginning with $r_{n,j}$ and working down to $r_{1,j}$, we may therefore assume that $e_i=1$ for $1\leq i\leq n-1$.
    
    Now let
    \begin{align*}
        v_1&=1,\\
        v_2&=x-r_1,\\
        v_3&=(x-r_1)(x-r_2)\\
        &\vdots\\
        v_n&=(x-r_1)\cdots(x-r_{n-1}),
    \end{align*}
    so that $B^\slant(f)=\{v_1,\ldots,v_n,\frac{f}{(x-r_n)^{e_n}},\ldots,\frac{f}{x-r_n}\}$. We then define constants $a_{i,j}\in k$ by
    \begin{align}\label{eq:v_i}
        v_i(x)\cdot(x-r_n)^{e_n}&=\sum_{j=1}^{N-e_n}a_{i,j}\cdot v_j(x)+R_i(x),
    \end{align}
    where $R_i(x)$ is a $k$-linear combination of the basis elements $\{\frac{f}{(x-r_n)^{e_n}},\ldots,\frac{f}{x-r_n}\}$. As matrices, we have
    \[P\cdot M=(a_{i,j})_{i,j=1}^n,\]
    so we need to show that $\det(a_{i,j})=\prod_{i=1}^{n-1}(r_i-r_n)^{e_n}$ (recall that we have assumed $e_i=1$ for $i<n$). Note that $R_i(r_\ell)=0$ for all $0\leq\ell<n$. Similarly, $v_i(r_\ell)=0$ for $i>\ell$. Substituting $x=r_\ell$ into Equation~\ref{eq:v_i} for $1\leq\ell<n$, we find that
    \[a_{i,j}=\begin{cases}(r_i-r_n)^{e_n} & i=j,\\ 0 & i<j.\end{cases}\]
    This implies that $\det(P\cdot M)=\prod_{i=1}^{n-1}(r_i-r_n)^{e_n}$ when $e_1=\ldots=e_{n-1}=1$, which completes the proof.
\end{proof}

\begin{rem}
If we loosen the requirement that $f$ be monic, we can still define and compute the duplicant of $f$. If $f\in k[x]$ with all roots $r_1,\ldots,r_n$ rational, then we can write $f=c\cdot h$, where $h=\prod_{i=1}^n(x-r_i)^{e_i}$ and $c\in k^\times$. The coefficient matrix $\Sigma(f)$ is now given by scaling each column of $\Sigma(h)$ by $c$, so we find that 
\begin{align*}
    \det\Sigma(f)&=c^{\rank\Sigma(h)}\cdot\det\Sigma(h)\\
    &=c^{\sum_i e_i}\cdot\det\Sigma(h).
\end{align*}
If we define $\D(f):=\det\Sigma(f)^2$ and denote $N:=\deg(f)=\deg(h)=\sum_{i=1}^n e_i$, then it follows from Theorem~\ref{thm:duplicant} that
\[\D(f)=c^{2N}\prod_{1\leq i<j\leq n}(r_i-r_j)^{2e_ie_j}.\]
\end{rem}

Unlike the usual discriminant, the duplicant need not vanish when $f$ has repeated roots. In fact, since $\D(f)$ is the square of the determinant of the monomial-to-Newton change-of-basis matrix, we have $\D(f)\neq 0$.

\section{Local-to-global principle, algebraically}\label{sec:algebraic local-to-global}
Our next goal is to derive an algebraic formula for the homotopical $D$-sum given in Theorem~\ref{thm:local-global homotopic}, which will yield a local-to-global principle for the algebraic unstable degree. We will begin by showing that this sum must be more subtle than the natural group structure on $\GW^u(k)$. To do so, we need to recall Cazanave's monoid operation on $[\mb{P}^1_k,\mb{P}^1_k]$ (whose group completion maps under $\deg^u$ to the standard group structure on $\GW^u(k)$) \cite[\S 3.1]{Caz12}. 

\begin{defn}\label{def:n sum}
Let $f$ be a polynomial with $\deg(f)=n$. Then there is a unique pair of polynomials $u,v$ with $\deg(u)\leq n-2$ and $\deg(v)\leq n-1$ satisfying the B\'ezout identity $fu+gv=1$. Given two pointed rational functions $f_1/g_1$ and $f_2/g_2$, let $u_i,v_i$ be the corresponding pairs of polynomials. Write
\[\begin{pmatrix}f_3 & -v_3\\g_3 & u_3\end{pmatrix}:=
\begin{pmatrix}f_1 & -v_1\\g_1 & u_1\end{pmatrix}
\begin{pmatrix}f_2 & -v_2\\g_2 & u_2\end{pmatrix}.\]
Then the \textit{na\"ive sum} is defined to be $f_1/g_1\nplus f_2/g_2:=f_3/g_3$, which is again a pointed rational function.
\end{defn}

By specifying the monoid structure on $[\mb{P}^1_k,\mb{P}^1_k]$ in Theorem~\ref{thm:cazanave}, Cazanave effectively gives a local-to-global principle for computing the unstable degree in terms of $\Bez^\mon$. However, we will see that this na\"ive local-to-global principle does not satisfy our desired criteria. The shortcoming is that when decomposing a pointed rational function $f/g$ by the na\"ive sum, the resulting ``local'' terms do not vanish at the same points as the original function $f/g$.

Instead, we will show that the local Newton matrix, namely our formula for the unstable local degree, satisfies a local-to-global principle with respect to the divisorial sum (see Definitions~\ref{def:D-sum} and~\ref{def:D-sum-alg}).

\subsection{Insufficiency of the na\"ive local-to-global principle}\label{sec:naive insufficient}
By Theorem~\ref{thm:cazanave}, one can express the unstable degree of a pointed rational function $f/g$ as a sum of unstable degrees of rational functions $f_1/g_1,\ldots,f_n/g_n$ of lesser degree. Iterating this process decreases the degrees of the na\"ive summands, so one can assume that each $f_i/g_i$ vanishes at a single point in $\mb{P}^1_k$. The unstable local degree of such a function should be equal to its unstable (global) degree, so this gives a \textit{na\"ive} local-to-global principle for the unstable degree. Unfortunately, the point of vanishing of $f_i/g_i$ can never belong to the vanishing locus of $f/g$:

\begin{prop}\label{prop:n-sum factor}
Let $f/g:\mb{P}^1_k\to\mb{P}^1_k$ be a pointed rational function. Assume that $f=f_1\cdot f_2$ for some non-constant polynomials $f_1,f_2$. Then there cannot exist $g_1,g_2$ such that $f_i/g_i:\mb{P}^1_k\to\mb{P}^1_k$ are pointed rational functions with $f/g=f_1/g_1\nplus f_2/g_2$.
\end{prop}
\begin{proof}
Suppose that $f/g=f_1/g_1\nplus f_2/g_2$ with $f=f_1\cdot f_2$. By definition of $\nplus$, we have $f=f_1f_2-v_1g_2$, so $v_1g_2=0$. Since $g_2$ is the denominator of a pointed rational function and the ring of polynomials over a field is a domain, we deduce that $v_1=0$. But this implies that $f_1u_1=1$, so $f_1$ is a unit. It follows that $f_1$ must be constant, contradicting our assumption that $f_1,f_2$ are non-constant.
\end{proof}

\begin{cor}\label{cor:no poincare-hopf}
Let $f/g:\mb{P}^1_k\to\mb{P}^1_k$ be a pointed rational function with vanishing locus $\{x_1,\ldots,x_n\}$. For each $x_i$, let $m_i$ be its minimal polynomial. Let $e_i\cdot\deg(m_i)$ be the order of vanishing of $f$ at $x_i$, so that $f=\prod_{i=1}^n m_i^{e_i}$. Then there cannot exist polynomials $g_1,\ldots,g_n$ such that $m_i^{e_i}/g_i:\mb{P}^1_k\to\mb{P}^1_k$ are pointed rational functions satisfying
\begin{align}\label{eq:no poincare-hopf}
\deg^u(f/g)=\sum_{i=1}^n\deg^u(m_i^{e_i}/g_i).
\end{align}
\end{cor}
\begin{proof}
By~\cite[Theorem 3.6]{Caz12}, finding $g_1,\ldots,g_n$ satisfying Equation~\ref{eq:no poincare-hopf} is equivalent to finding $g_1,\ldots,g_n$ such that
\[\frac{f}{g}=\frac{m_1^{e_1}}{g_1}\nplus\cdots\nplus\frac{m_n^{e_n}}{g_n}.\]
Since $\nplus$ is associative, we can reduce via induction to the $n=2$ case. It now follows from Proposition~\ref{prop:n-sum factor} that such a factorization cannot exist.
\end{proof}

Corollary~\ref{cor:no poincare-hopf} tells us that the B\'ezoutian with respect to the monomial basis will not give a satisfactory \textit{unstable local} degree, in contrast with the unstable \textit{global} degree~\cite{Caz12} and the \textit{stable} local degree~\cite{BMP21}. This is because the local terms in any na\"ive decomposition will not vanish at any points in the vanishing locus of our original function.

\subsection{Divisorial sums of local terms}
We have just seen that in general, the na\"ive sum will not give us a satisfactory local-to-global principle. In Theorem~\ref{thm:local-global homotopic}, we saw that our desired local-to-global principle requires that we work with the homotopical sum $(-)\circ\cv_D$, where $D$ is the vanishing locus of the pointed rational map that we are trying to decompose. In contrast, the na\"ive sum is defined homotopically by collapsing the complement of the locus $\{0,\infty\}$. In other words, the na\"ive sum fails to give the desired local-to-global principle, because the vanishing locus of a pointed rational map is generally not a subset of $\{0,\infty\}$.

Our next goal is to compute the image in $\GW^u(k)$ (under the B\'ezoutian) of the addition law $\sum_D:=(-)\circ\cv_D$. We will also call this image the (algebraic) $D$-sum, denoted $\oplus_D$, which will depend on $D$. We will use Theorem~\ref{thm:local-global homotopic}, our formula for the unstable local degree, and Cazanave's formula for the unstable global degree to compute $\oplus_D$.

\begin{defn}\label{def:D-sum-alg}
Let $D=\{r_1,\ldots,r_n\}\subset\mb{P}^1_k(k)$ be a finite set of rational points. The \emph{(algebraic) $D$-sum} is the function
\[\bigoplus_D:\bigoplus_{i=1}^n\GW^u(k)\to\GW^u(k)\]
satisfying $\bigoplus_D\deg^u(f_i)=\deg^u(\vee_i f_i\circ\cv_D)$ for any $n$-tuple $f_1,\ldots,f_n:\mb{P}^1_k\to\mb{P}^1_k$ of pointed rational maps. In other words, $\bigoplus_D:=\deg^u\circ\sum_D\circ(\deg^u)^{-1}$.
\end{defn}

As with the homotopical $D$-sum, we will say \emph{(algebraic) divisorial sum} when we do not wish to specify the divisor $D$.

Our goal is to give an algebraic formula for $\oplus_D$ for any $n$-tuple of elements in $\GW^u(k)$. As a first step, we can use the homotopical local-to-global principle (Theorem~\ref{thm:local-global homotopic}) and our formula for the unstable local degree (Lemma~\ref{lem:degree=newton matrix}) to compute a formula for $\oplus_D$ in some cases.

\begin{prop}\label{prop:local-to-global-rational-function}
    Let $f/g:\mb{P}^1_k\to\mb{P}^1_k$ be a pointed rational function, with vanishing locus $D=\{r_1,\ldots,r_n\}$. Let $\deg^u_{r_i}(f/g)=(\beta_i,d_i)\in\GW^u(k)$, and let $m_i = \rank \beta_i$ and $m=\sum_i m_i$. Then
    \begin{equation}\label{eq:d-sum-local-to-global}
    \bigoplus_D((\beta_1,d_1),\ldots,(\beta_n,d_n))=\big(\bigoplus_{i=1}^n\beta_i,\prod_{i=1}^n d_i\cdot\prod_{i<j}(r_i -r_j)^{2m_i m_j}\big).
    \end{equation}
\end{prop}
\begin{proof}
As described in Section~\ref{sec:unstable gw}, we can describe elements of $\GW^u(k)$ in terms of $k$-vector space, a choice of basis, and the Gram matrix of a symmetric bilinear form with respect to that basis. Recall the notation $Q(f/g):=k[x,\frac{1}{g}]/(\frac{f}{g})$, and consider the following bases of $Q(f/g)$:
\begin{align*}
    B^\mon(f)&=\bigg\{1,x,x^2,\ldots,x^{m-1}\bigg\},\\
     B^{\mon/g}(f)&=\bigg\{\frac{1}{g},\frac{x}{g},\frac{x^2}{g},\ldots,\frac{x^{m-1}}{g}\bigg\},\\
     B^{\Nwt}(f) &= \bigcup_{i=1}^n\bigg\{ \frac{f(x)}{(x-r_i)},\frac{f(x)}{(x-r_i)^2}, \ldots,  \frac{f(x)}{(x-r_i)^{m_i}} \bigg\},\\
     B^{\Nwt/g}(f) &= \bigcup_{i=1}^n\bigg\{ \frac{f(x)}{(x-r_i)g(x)},\frac{f(x)}{(x-r_i)^2 g(x)}, \ldots,  \frac{f(x)}{(x-r_i)^{m_i}g(x)} \bigg\}.
    \end{align*}
By \cite[Theorem~3.6]{Caz12}, the Gram matrix of $\deg^u(f/g)$ with respect to $B^\mon(f)$ is given by $\Bez^\mon(f/g)=(a_{ij})$, where
\begin{equation}\label{eq:cazanave bezoutian}
\frac{f(x)g(y)-f(y)g(x)}{x-y}=\sum_{i,j}a_{ij}x^iy^j.
\end{equation}
Dividing both sides of Equation~\ref{eq:cazanave bezoutian} by $g(x)g(y)$, we obtain
\begin{equation*}
\frac{f(x)/g(x)-f(y)/g(y)}{x-y}=\sum_{i,j}a_{ij}\frac{x^i}{g(x)}\frac{y^j}{g(y)}.
\end{equation*}
As described in \cite[Section 3 and Equation (22)]{KW20}, the (global) Newton matrix is given by $\Nwt(f/g)=(a_{ij})$, where
\[\frac{f(x)/g(x)-f(y)/g(y)}{x-y}=\sum_{i,j}a_{ij}v_i(x)v_j(y)\]
and $\{v_1,\ldots,v_{m-1}\}=B^{\Nwt/g}(f)$. This means that $\deg^u(f/g)$ and $\Nwt(f/g)$ are related by changing basis from $B^{\mon/g}(f)$ to $B^{\Nwt/g}(f)$. Note that the relevant change-of-basis matrix is identical to the change-of-basis matrix $T^\mon_\Nwt$ from $B^\mon(f)$ to $B^\Nwt(f)$. Moreover, $\Nwt(f/g)=\bigoplus_{i=1}^n\Nwt_{r_i}(f/g)$ by \cite[Definition 7]{KW20}. In summary, we find that
\begin{align*}
    \deg^u(f/g)&=(T^\mon_\Nwt)^\intercal\cdot\Nwt(f/g)\cdot T^\mon_\Nwt\\
    &=(T^\mon_\Nwt)^\intercal\cdot\bigoplus_{i=1}^n\Nwt_{r_i}(f/g)\cdot T^\mon_\Nwt\\
    &=(T^\mon_\Nwt)^\intercal\big(\sum_{i=1}^n\deg^u_{r_i}(f/g)\big)T^\mon_\Nwt,
\end{align*}
where the last equality follows from Lemma~\ref{lem:degree=newton matrix}. We conclude the proof by taking determinants and recalling that $\det(T^\mon_\Nwt)^2=\prod_{i<j}(r_i-r_j)^{2m_im_j}$ by Theorem~\ref{thm:duplicant}.
\end{proof}

\begin{rem}
    Note that we can rewrite Equation~\ref{eq:d-sum-local-to-global} as
    \[\bigoplus_D((\beta_1,d_1),\ldots,(\beta_n,d_n))=\big(\bigoplus_{i=1}^n\beta_i,\D(f)\cdot\prod_{i=1}^n d_i\big),\]
    where $\D(f)$ is the duplicant of $f$.
\end{rem}

\begin{rem}
As written, Proposition~\ref{prop:local-to-global-rational-function} does not come close to the generality we are aiming for in Theorem~\ref{thm:D-sum-all-f_i}. Indeed, not all pointed maps $\P^1\to\P^1$ in the unstable $\A^1$-homotopy category are represented by rational functions. Moreover, not every element $(\beta,d)\in\GW^u(k)$ arises as the unstable local degree of some rational function $f/g$, even when $\rank\beta\geq 0$ \cite{Quick-Strand-Wilson}. The paper \cite{Barth-Hornslien-Quick-Wilson} studies endomorphisms of $\P^1$ with unstable degree of negative rank explicitly by lifting to the Jouanolou device, but this again fails to give, for an arbitrary $(\beta,d)\in\GW^u(k)$, an explicit pointed map $\P^1\to\P^1$ whose unstable local degree at some point is $(\beta,d)$. We will bridge the gap between Proposition~\ref{prop:local-to-global-rational-function} and Theorem~\ref{thm:D-sum-all-f_i} in the next section by using Proposition~\ref{prop:local-to-global-rational-function} as the base case of an induction argument. 
\end{rem}

\section{Proof of Theorem \ref{thm:D-sum-all-f_i}}\label{sec:general case}
It is a theorem of Morel \cite[Section 7.3, 7.26]{Mor12} that the Hopf map $\eta: \A^2 - \{0\} \to \P^1$ defined $\eta(x,y)= [x,y]$ and the map $\iota_{1,\infty}: \P^1 \to \P^{\infty}$ classifying $\mathcal{O}(1)$ form an $\A^1$-fiber sequence
\begin{equation}\label{eq:central for pi1}
\A^2 - \{0\} \to \P^1 \to \P^{\infty},
\end{equation}

which induces a central extension
\begin{equation}\label{eq:Kmw2pA1Gm_extension}
1 \to \Kmw_2 \cong  \pA (\A^2 - \{0\})  \to \pA(\P^1) \stackrel{\wp}{\to} \pA( \P^{\infty}) \cong \G_m \to 1.
\end{equation} 

We will base our schemes at $(1,0) \in \A^2 - \{0\} $ and the corresponding images under the appropriate maps.

\begin{notn}
    We will frequently need to use the inverse of the isomorphism $\deg^u$ to convert from elements in $\GW^u(k)$ to their associated elements in $[\P^1,\P^1]$. To simplify notation, we will define
    \[\phi:=(\deg^u)^{-1}:\GW^u(k)\to[\P^1,\P^1].\]
\end{notn} 

\begin{lem}\label{lem:wppabeta}
Let $\beta^u=(\beta,d)$ in $\GW^u(k)$ have rank $m$. Then $\wp \circ \pA (\phi\beta^u) = m \circ \wp$, where $m$ denotes the map $\G_m \to \G_m$ given by $z \mapsto z^m$.
\end{lem}

\begin{proof}
Given a sheaf of pointed sets $\mc{F}$, the \emph{contraction} of $\mc{F}$ is the sheaf
\[\mc{F}_{-1}:=\underline{\mr{Map}}(\mb{G}_m,\mc{F}).\]
See \cite[Theorem 4.37 p.125]{Voev_Sus_Fr_Cycles_tr_motivic_homology}, \cite[Remark 2.23]{Mor12} or \cite[Section 4]{Bac24}. Morel computes the unstable $\A^1$-homotopy classes $[\P^1, \P^{n}]$ by 
\[
[\P^1, \P^{n}] \cong [\Sigma \G_m, \P^n] \cong \pA(\P^n)_{-1}(k).
\] Furthermore, Morel shows the $\G_m$-torsor $\A^{n+1} - \{0\} \to \P^n$ is the universal cover, whence $\pA(\P^n) \cong \G_m$ (for $n\geq 2$) and $\pA(\P^n)_{-1} \cong \mb{Z}$. Thus $[\P^1, \P^{\infty}] \cong \colim_n [\P^1, \P^{n}] \cong \mb{Z}$. An explicit isomorphism $[\P^1, \P^{\infty}] \to \mb{Z}$ sends a map in $[\P^1, \P^{\infty}]$ to the degree of the corresponding pullback of $\mathcal{O}(1)$.

The composition $\P^1 \xrightarrow{\phi\beta^u} \P^1 \xrightarrow{\iota_{1,\infty}} \P^{\infty}$ classifies $\mathcal{O}(m)$. So does $\P^1  \stackrel{\iota_{1,\infty}}{\longrightarrow} \P^{\infty} \stackrel{Bm}{\to} \P^{\infty}$, where $Bm$ is defined by $B \G_m \simeq \P^{\infty}$ and the map $m:\G_m \to \G_m$. Thus $Bm \circ \iota_{1,\infty} \simeq  \iota_{1,\infty} \circ \phi\beta^u $ by Morel's computation $[\P^1, \P^{\infty}] \cong \mb{Z}$. To conclude, note $\pA (Bm \circ \iota_{1,\infty} ) = m \circ \wp$ and $\pA(\iota_{1,\infty} \circ \phi\beta^u) = \wp \circ \pA(\phi\beta^u)$.
\end{proof}

Let $\iota_j : \P^1 \to \bigvee_{i=1}^n \P^1$ denote the inclusion of the $j$th summand. Let $s_j: \bigvee_{i=1}^n \P^1 \to \P^1$ denote the map which is the identity on the $j$th summand and the constant map to the point on the other summands.

\begin{defn}
A map $c: \P^1 \to \bigvee_{i=1}^n \P^1$ with the property that $s_j \circ c$ is equivalent to the identity map on $\P^1$
\[
s_j \circ c \simeq 1_{\P^1}
\] for $j=1,\ldots,n$ will be said to be a \emph{good pinch map}. 
\end{defn}

\begin{ex}[Simplicial pinch]\label{ex:c+}
The standard isomorphism $\P^1 \xrightarrow{\simeq} \Sigma \G$ (from $\P^1 \cong \A^1 \cup_{\G_m} \A^1$) and the standard pinch map $S^1 \to S^1 \vee S^1$ define a good pinch map $c_{+}:\P^1 \to  \P^1 \vee \P^1$. By iterating this construction, we more generally obtain a good pinch map
\[c_+:\P^1\to\bigvee_{i=1}^n\P^1.\]
\end{ex}

\begin{ex}[Divisorial pinch]\label{ex:CD}
$\cv_D: \P^1 \to \P^1/(\P^1 - D) \stackrel{\simeq}{\leftarrow} \bigvee_{i = 1}^n \P^1$  is a good pinch map for $D = \{ r_1, \ldots, r_n\}$ with $r_i$ in $k$ and $r_i \neq r_j$ for $i \neq j$. 
\end{ex}

\begin{defn}\label{df:a}
Let $a$ denote the left adjoint to the inclusion of strongly $\A^1$-invariant sheaves of groups on the Nisnevich site into sheaves of groups \cite[p. 184]{Mor12}. 
\end{defn}

\begin{lem}\label{lem:pi1deg0}
Let $c$ be a good pinch map. Let $\beta^u=(\beta,d)$ and $\beta^u_0=(\beta_0,d_0)$ be elements of $\GW^u(k)$ with $\beta_0$ of rank $0$. Then
\begin{enumerate}
\item \label{it:pi1deg0:landsK2} $\pA(\phi\beta^u_0)$ factors through $\pA(\eta)$, so that we obtain a map of the form $\pA(\phi\beta^u_0): \pA(\P^1) \to \Kmw_2$.
\item \label{it:pi1deg0:additive}  For all $U$ in $\Sm_k$ and all $\gamma$ in $ \pA(\P^1)(U)$, we have \[
\pA((\phi\beta^u \vee \phi\beta^u_0) \circ c )(\gamma)= \pA(\phi\beta^u)(\gamma)\pA(\phi\beta^u_0)(\gamma).
\] In particular, 
\[
\pA(\phi(\beta^u+\beta^u_0))(\gamma) = \pA(\phi\beta^u)(\gamma)\pA(\phi\beta^u_0)(\gamma).
\]
\end{enumerate}
\end{lem}

\begin{proof}
\eqref{it:pi1deg0:landsK2}: Lemma~\ref{lem:wppabeta} implies that the image of $\pA(\phi\beta^u_0)$ lies in the kernel of $\wp$, which is the image of $\pA(\eta) \cong \Kmw_2$ by \eqref{eq:Kmw2pA1Gm_extension}. 

\eqref{it:pi1deg0:additive}: By Morel's $\A^1$-version of the Seifert--van Kampen Theorem \cite[Theorem 7.12]{Mor12}, the canonical maps $\pA(\iota_j)$ assemble to an isomorphism  
\begin{equation}\label{eq:MVwedgeP1}
 a(*_{i=1}^n \pA(\P^1)) \xrightarrow{\cong} \pA( \bigvee_{i=1}^n \P^1)  
\end{equation} from the initial strongly $\A^1$-invariant sheaf $a(*_{i=1}^n \pA(\P^1))$ on the free product $*_{i=1}^n \pA(\P^1)$ to $\pA( \bigvee_{i=1}^n \P^1)$.

Since $\Kmw_2$ lies in the center of $\pA(\P^1)$, there is an induced addition homomorphism $\pA(\P^1) \times  \Kmw_2 \stackrel{+}{\to} \pA(\P^1)$.

By \eqref{it:pi1deg0:landsK2}, the map $\pA(\phi\beta^u) * \pA(\phi\beta^u_0): \pA(\P^1) * \pA(\P^1) \to \pA(\P^1)$ factors through  $\pA(\P^1) \times  \Kmw_2 \stackrel{+}{\to} \pA(\P^1)$. Since  $\pA(\P^1) \times  \Kmw_2 $ is strongly $\A^1$-invariant \cite [Theorem 1.9, Theorem 3.37]{Mor12}, the resulting map $\pA(\P^1) * \pA(\P^1) \to \pA(\P^1) \times  \Kmw_2$ factors through the canonical map $*_{i=1}^2 \pA(\P^1) \to a(*_{i=1}^2 \pA(\P^1))$. In total, we obtain a map $g:a(*_{i=1}^2 \pA(\P^1)) \to \pA(\P^1) \times  \Kmw_2$ induced by $\pA(\phi\beta^u) * \pA(\phi\beta^u_0)$. Since $c$ is a good pinch map, $g$ is identified with $\pA(\phi\beta^u \circ s_1) \times \pA (\phi\beta^u_0 \circ s_2) $ by the isomorphism \eqref{eq:MVwedgeP1}. It follows that $\pA(\phi\beta^u \vee \phi\beta^u_0) = + \circ ( \pA(\phi\beta^u \circ s_1) \times \pA (\phi\beta^u_0 \circ s_2))$, showing that \[
\pA((\phi\beta^u \vee \phi\beta^u_0) \circ c )(\gamma)= \pA(\phi\beta^u)(\gamma)\pA(\phi\beta^u_0)(\gamma) \] as claimed. By \cite[Proposition 3.23, Corollary 3.10, Theorem 3.22]{Caz12}, we have $(\phi\beta^u\vee\phi\beta^u_0)\circ c_+=\phi(\beta^u+\beta^u_0)$ in $[\P^1,\P^1]$, showing \eqref{it:pi1deg0:additive}.
\end{proof}

\begin{rem}\label{rem:commutators_in_2nil_groups}
For clarity, we include the following remark on $2$-nilpotent groups. Let
\begin{equation}\label{eq:2-nil-extension}
1 \to K \to G \to A \to 1
\end{equation} be a central extension of groups with $K$ and $A$ abelian. Let for all $g_1, g_2$ in G, let $[g_1,g_2] = g_1 g_2 g_1^{-1} g_2^{-1}$ denote the commutator. The commutator determines a group homomorphism $\commu: A \otimes A \to K$ defined by taking $a_1 \otimes a_2$ to $[g_1,g_2]$ where $g_i \mapsto a_i$. 
\end{rem}

For a sheaf of groups $G$, let $G \to G^{\tnil}$ denote the initial map to a sheaf of 2-nilpotent groups and let $G \to G^{\ab}$ denote the abelianization. The canonical map 
\[*_{i=1}^n \pA( \P^1 ) \to \times_{i=1}^n \pA(\P^1)\] factors through $*_{i=1}^n \pA( \P^1 ) \to (*_{i=1}^n \pA( \P^1 ))^{\tnil}$ because $\pA(\P^1)$ is $2$-nilpotent. Since $ \times_{i=1}^n \pA(\P^1)$ is strongly $\A^1$-invariant, we obtain a map $a( (*_{i=1}^n \pA( \P^1 ))^{\tnil}) \to \times_{i=1}^n \pA(\P^1)$. Since $\times_{i=1}^n \pA(\P^1)$ is $2$-nilpotent, we obtain a futher map
\[r:(a( (*_{i=1}^n \pA( \P^1 ))^{\tnil}))^{\tnil}\to \times_{i=1}^n \pA(\P^1).\]
 Let $K$ denote the kernel. We claim $K$ is ``generated by commutators" in the following sense. 

\begin{lem}\label{lem:Kr_central}
The extension
\[
1 \to K \to (a( (*_{i=1}^n \pA( \P^1 ))^{\tnil}))^{\tnil} \stackrel{r}{\longrightarrow} \times_{i=1}^n \pA(\P^1) \to 1
\] is central and $K$ receives a surjection 
\[
\times_{i \neq j} a(\pA( \P^1 )^{\ab}\otimes \pA(\P^1)^{\ab}) \to K
\] 
 given by summing maps $$ a(\pA( \P^1 )\otimes \pA(\P^1)) \to K$$ for $i \neq j$ defined by sending $\gamma_1 \otimes \gamma_2$ in $\pA( \P^1 )(U)\otimes \pA(\P^1)(U)$ to the commutator $$[\pA(\iota_i)(\gamma_1), \pA(\iota_j)(\gamma_2)]$$ in $(a( (*_{i=1}^n \pA( \P^1 ))^{\tnil}))^{\tnil}$. 
\end{lem}

\begin{proof}
For any sheaf of $2$-nilpotent groups $G$, we have a pushout diagram
\begin{equation}\label{eq:central_extension_G2nil}
\begin{tikzcd}
\oplus_{i \neq j} (G^{\ab} \otimes G^{\ab})\arrow[rr]\arrow[d] &&(*_{i =1}^n G)^{\tnil}\arrow[d] \\
1 \arrow[rr]  && \times_{i =1}^n G
\end{tikzcd}   
\end{equation} where the top horizontal row is given by commutators $[\iota_i(-), \iota_j(-)] $ and where the right vertical arrow is an epimorphism inducing a central extension. To see this, note that   $*_{i=1}^n G \to \times_{i=1}^n G$ is an epimorphism, whence so is 
\[
 (*_{i=1}^n G)^{\tnil} \to \times_{i=1}^n G.
 \]
The abelianization of $*_{i=1}^n G$ factors through $\times_{i=1}^n G$, whence we have \[
(*_{i=1}^n G)^{\tnil} \to \times_{i=1}^n G\stackrel{\theta_1} {\to} (*_{i=1}^n G)^{\ab}
\] with $\theta_1$ an epimorphism. It follows that $(*_{i=1}^nG)^{\ab} \cong \times_{i=1}^n G^{\ab}$. The pushout
\begin{equation*}
\begin{tikzcd}
\oplus_{i ,j} (G^{\ab} \otimes G^{\ab})\arrow[rr]\arrow[d] &&(*_{i =1}^n G)^{\tnil}\arrow[d] \\
1 \arrow[rr]  && \times_{i =1}^n G^{\ab}
\end{tikzcd}  
\end{equation*} produces the claimed pushout \eqref{eq:central_extension_G2nil}. 

Since $a$ is a left adjoint, $a$ preserves epimorphisms and pushouts, so we may apply $a$ to \eqref{eq:central_extension_G2nil} with $G = \pA(\P^1)$, and obtain a pushout. Let $K'$ denote the image of the map
\[
a(\oplus_{i \neq j} (\pA(\P^1)^{\ab} \otimes \pA(\P^1)^{\ab})) \to a( (*_{i=1}^n \pA( \P^1 ))^{\tnil})
\] producing the extension
\begin{equation}\label{eq:K'atnil_ext}
1 \to K' \to a( (*_{i=1}^n \pA( \P^1 ))^{\tnil}) \to a(\times_{i=1}^n \pA(\P^1)) \cong \times_{i=1}^n \pA(\P^1) \to 1
\end{equation} where the last isomorphism follows because $\pA(\P^1)$ is strongly $\A^1$-invariant. Moreover $K'$ receives a surjection from $a(\oplus_{i \neq j} (\pA(\P^1)^{\ab} \otimes \pA(\P^1)^{\ab})) $. The extension \eqref{eq:K'atnil_ext} surjects onto the extension
\begin{equation*}
1 \to K \to (a( (*_{i=1}^n \pA( \P^1 ))^{\tnil}))^{\tnil} \to \times_{i=1}^n \pA(\P^1) \to 1,
\end{equation*} whence $a(\oplus_{i \neq j} (\pA(\P^1)^{\ab} \otimes \pA(\P^1)^{\ab})) \to K$ is a surjection as claimed. Since commutators are central in 2-nilpotent extensions, $K$ is central in $(a( (*_{i=1}^n \pA( \P^1 ))^{\tnil}))^{\tnil} $ as claimed.

\end{proof}

\begin{rem}\label{rem:2nil_quotient_free}
   The group sheaf $(a( (*_{i=1}^n \pA( \P^1 ))^{\tnil}))^{\tnil}$ deserves some comment. Let $G = \pA(\P^1)$. The innermost $2$-nilpotent quotient $(*_{i=1}^n G)^{\tnil}$ serves to produce a surjection 
   \[
   \oplus_{i \neq j} (G^{\ab} \otimes G^{\ab} ) \to \ker ((*_{i=1}^n G)^{\tnil} \to \oplus_{i=1}^n G).
   \] 
    Passing to the initial strongly $\A^1$-invariant sheaf $a((*_{i=1}^n G)^{\tnil})$ produces a map $$\pA(\vee_{i=1}^n \P^1) \to a((*_{i=1}^n G)^{\tnil}).$$ The final $2$-nilpotent quotient $(a((*_{i=1}^n G)^{\tnil}))^{\tnil}$ serves to make the quotient map 
   \[(a((*_{i=1}^n G)^{\tnil}))^{\tnil} \to \oplus_{i=1}^n G\]
    the quotient of a central extension. 
\end{rem}

Let $\rho$ denote the canonical map  \[
 \rho: \pA(\bigvee_{i=1}^n \P^1)\cong a (*_{i=1}^n \pA( \P^1 )) \to (a( (*_{i=1}^n \pA( \P^1 ))^{\tnil}))^{\tnil},
 \]
where the isomorphism $\pA(\bigvee_{i=1}^n \P^1)\cong a(*_{i=1}^n\pA(\P^1))$ is Morel's $\A^1$-version of the Seifert--van Kampen theorem (Equation~\ref{eq:MVwedgeP1}). We show that $\pA(\vee_i f_i\circ c)$ factors through $\rho$ when $c$ is a good pinch map.

\begin{prop}\label{prop:veefic_factors_2_nil}
Let $c: \P^1 \to \bigvee_{i=1}^n \P^1$ be a good pinch map. Let $f_i: \P^1 \to \P^1$ be endomorphisms in unstable $\A^1$-homotopy theory for $i = 1,\ldots, n$. Then 
\[
\pA(\vee f_i \circ c): \pA(\P^1) \to \pA(\P^1)
\]
 factors as a composition  
 \[
 \pA(\P^1) \xrightarrow{\rho \circ \pi_1^{\A^1}(c)}  (a( (*_{i=1}^n \pA( \P^1 ))^{\tnil}))^{\tnil} \xrightarrow{\overline{\pA(\vee f_i)}} \pA(\P^1)
 \]
 where $\overline{\pA(\vee f_i)}$ is the unique map such that $\pA(\vee f_i) =\overline{\pA(\vee f_i)} \circ \rho$. 
\end{prop}

\begin{proof}
Note that $\pA(\vee f_i ): \pA( \bigvee_{i=1}^n \P^1) \to \pA(\P^1)$. By Morel's $\A^1$-version of the Seifert--van Kampen Theorem (\cite[Theorem 7.12]{Mor12}, recalled in Equation~\ref{eq:MVwedgeP1}), the map $*_{i=1}^n \pA(\iota_i): *_{i=1}^n \pA (\P^1) \to \pA( \bigvee_{i=1}^n \P^1)$ induces an isomorphism $a(*_{i=1}^n \pA (\P^1) ) \stackrel{\cong}{\to }  \pA( \bigvee_{i=1}^n \P^1)$. 

Under this isomorphism  $\pA(\vee f_i )$ is identified with the map induced by $*_{i=1}^n \pA(f_i): *_{i=1}^n \pA (\P^1) \to \pA(\P^1)$. Since $ \pA(\P^1)$ is $2$-nilpotent, $*_{i=1}^n \pA(f_i)$ factors as
\[
*_{i=1}^n \pA (\P^1) \to (*_{i=1}^n \pA (\P^1))^{\tnil} \to \pA(\P^1).
\] 
 Since $ \pA(\P^1)$ is strongly $\A^1$-invariant, we obtain the factorization
\[
*_{i=1}^n \pA (\P^1) \to a( (*_{i=1}^n \pA (\P^1))^{\tnil}) \to \pA(\P^1),
\] giving the commutative diagram 

\[\begin{tikzcd}
    && \pA\big(\bigvee_{i=1}^n\mb{P}^1\big)\arrow[d,"\cong"]\arrow[drr,bend left=15,"\pA(\vee_i f_i)"]&&\\
   \pA(\mb{P}^1)\arrow[urr,bend left=15,"\pA(c)"]\arrow[drr,bend right=15] && a\big(*_{i=1}^n\pA(\mb{P}^1)\big)\arrow[d]\arrow[rr,"a(*_{i=1}^n\pA(f_i))"] && \pA(\mb{P}^1)\\
    && a( (*_{i=1}^n \pA (\P^1) )^{\tnil})\arrow[urr,bend right=15,'] &&
\end{tikzcd}.\]
Since $\pi_1^{\A^1}(\P^1)$ is $2$-nilpotent, the claimed factorization follows.
\end{proof}

The map $\wp:\pA(\P^1) \to \G_m$ in the central extension \eqref{eq:Kmw2pA1Gm_extension} admits a section 
\[
\theta: \G_m \to \pA(\P^1)
\] coming from the map $\G_m \to \Omega \Sigma \G_m \simeq_{\A^1} \Omega \P^1$. Since $\pA(\P^1)$ is a sheaf of $2$-nilpotent groups, there is an induced map
\begin{equation}\label{eq:commu_def}
\commu: \G_m \otimes \G_m \to \Kmw_2
\end{equation} from the bilinear map $\G_m \times \G_m \to \Kmw_2$ sending $\alpha_1 \otimes \alpha_2$ in $\G_m(U) \otimes \G_m(U)$ to the commutator $[\theta(\alpha_1), \theta(\alpha_2)]$ in $\pA(\P^1)(U)$ for any $U \in \Sm_k$. Compare with Remark~\ref{rem:commutators_in_2nil_groups}.

\begin{lem}\label{overlineveefi_on_K}
Let $f_1,\ldots,f_n: \P^1 \to \P^1$ be endomorphisms of $\P^1$ in unstable $\A^1$-homotopy theory and let $m_i:=\rank\deg^u(f_i)$. The restriction of the morphism $\overline{\pA(\vee f_i)}$ to $K$ fits in the commutative diagram

\begin{equation*}
\begin{tikzcd}
\times_{i \neq j} a(\pA( \P^1 )^{\ab}\otimes \pA(\P^1)^{\ab}) \arrow[rr,twoheadrightarrow] \arrow[d,"\times (\wp \otimes \wp)"] && K \arrow[rr,"\overline{\pA(\vee f_i)}"]&& \pA(\P^1)\\
\times_{i \neq j} \G_m \otimes \G_m \arrow[rr,"\times_{i \neq j} (m_i \otimes m_j)"] &&\times_{i \neq j} \G_m \otimes \G_m \arrow[rr,"\times_{i \neq j} \commu"] && \Kmw_2. \arrow[u]
\end{tikzcd}  
\end{equation*} 
\end{lem}

\begin{proof}
 We must show two maps $\times_{i \neq j} a(\pA( \P^1 )^{\ab}\otimes \pA(\P^1)^{\ab}) \to \pA(\P^1)$ are equal. Since $\pA(\P^1)$ is strongly $\A^1$-invariant, it suffices to show their precompositions with $\times_{i \neq j} (\pA( \P^1 )^{\ab}\otimes \pA(\P^1)^{\ab}) \to \times_{i \neq j} a(\pA( \P^1 )^{\ab}\otimes \pA(\P^1)^{\ab})$ are equal. Note that the product $\times_{i \neq j} (\pA( \P^1 )^{\ab}\otimes \pA(\P^1)^{\ab})$  is canonically isomorphic the the sum
\[\oplus_{i \neq j} (\pA( \P^1 )^{\ab}\otimes \pA(\P^1)^{\ab}).\]
 Fix $U \in \Sm_k$ and $i<j$. For $\gamma_1$ and $\gamma_2$ in $\pA(\P^1)(U)$, the tensor $\gamma_1 \otimes \gamma_2$ determines an element $(\gamma_1 \otimes \gamma_2)_{i,j}$ of $\oplus_{i \neq j} (\pA( \P^1 )^{\ab}\otimes \pA(\P^1)^{\ab})$ in the $(i,j)$th summand which we then map into $K$ producing an element we call $(\gamma_1 \otimes \gamma_2)_{i,j,K}$. Applying $\overline{\pA(\vee f_i)}$, we compute
\[
\overline{\pA(\vee f_i)} (\gamma_1 \otimes \gamma_2)_{i,j,K} = [\pA(f_i)(\gamma_1), \pA(f_j)(\gamma_2)] = \commu(\wp(\gamma_1)^{m_1}, \wp(\gamma_2)^{m_2})
\] where the last equality follows by Lemma~\ref{lem:wppabeta} and the existence of the commutator map $\commu$ given in \eqref{eq:commu_def}. 
\end{proof}

\begin{lem}\label{lem:Deltac12_to_K}
Let $c_1$ and $c_2$ be good pinch maps. Then $(\rho \circ \pA (c_1))(\rho \circ \pA (c_2))^{-1} $ determines a homomorphism of sheaves of groups 
\[
\Delta_{c_1,c_2}:\pA(\P^1) \to K.
\]
\end{lem}

\begin{proof}
Mapping $\gamma \in \pA(\P^1)(U)$ to $(\rho \circ \pA (c_1))(\gamma)(\rho \circ \pA (c_2))^{-1}(\gamma) $ determines a map of sheaves of sets $$\pA(\P^1) \to (a( (*_{i=1}^n \pA( \P^1 ))^{\tnil}))^{\tnil}.$$ Since $c_1$ and $c_2$ are good pinch maps, the composition of $\pA(c_i)$ with the map $r: (a( (*_{i=1}^n \pA( \P^1 ))^{\tnil}))^{\tnil} \to \times_{i=1}^n \pA(\P^1)$ is $(1,\ldots, 1)$ for $i=1,2$. Thus $r \circ (\rho \circ \pA (c_1)(\rho \circ \pA (c_2))^{-1} ) = 0$. Therefore by Lemma~\ref{lem:Kr_central}, we have a map of sheaves of sets
\[
\Delta_{c_1,c_2} = \rho \circ \pA (c_1)(\rho \circ \pA (c_2))^{-1}: \pA(\P^1) \to K.
\] For notational simplicity let $g= \rho \circ \pA (c_1)$ and $h = \rho \circ \pA (c_2)$. For any $U$ in $\Sm_k$ and $\gamma_1,\gamma_2$ in $\pA(\P^1)(U)$, we have that 
\begin{align*}
g(\gamma_1 \gamma_2) h(\gamma_1 \gamma_2)^{-1} &= 
g(\gamma_1 ) (g(\gamma_2)h(\gamma_2)^{-1})h(\gamma_1)^{-1} \\ &=
(g(\gamma_1)h(\gamma_1)^{-1})(g(\gamma_2)h(\gamma_2)^{-1}),
\end{align*} 
whence $\Delta_{c_1,c_2}$ is a homomorphism as claimed.

\end{proof}

Fix $D = \{ r_1, \ldots, r_n \}$ with $r_i \in k$ and $r_i \neq r_j$. For endomorphisms $f_1,\ldots,f_n$ of $\P^1$, define 
\begin{align*}
    \sum_{D^\alg}(f_1,\ldots, f_n) &:= \phi\big(\bigoplus_{i=1}^n\beta_i,\prod_{i=1}^n d_i\cdot\prod_{i<j}(r_i-r_j)^{2m_im_j}\big)
\end{align*}
where $\deg^u(f_i)=(\beta_i, d_i)$ with $m_i = \rank \beta_i$. When there is no danger of confusion, we will use the abbreviations $\sum_D f_i$ and $\sum_{D^\alg}f_i$. 

For integers $m_1,\ldots, m_n$, let $\beta_{m_1,\ldots,m_n,D}$ denote the element of $\GW^u(k)$ given by $\beta_{m_1,\ldots,m_n,D} = (0, \prod_{i<j}(r_i -r_j)^{2 m_i m_j})$. 

\begin{lem}\label{lem:p1D-pi1Dalg=fiDeltaBeta}
Let $f_1,\ldots,f_n$ be endomorphisms of $\P^1$ in unstable $\A^1$-homotopy theory and let $m_i = \rank \deg^u f_i$. We have an equality
\begin{align}\label{eq:pAD-pADalgf}
 (\overline{\pA(\vee f_i)}\circ \Delta_{\cv_D,c_+})\pA(\phi\beta_{m_1,\ldots,m_n,D}) = \pA(\sum_Df_i)(\pA\sum_{D^\alg}f_i)^{-1}.
\end{align} 
of maps of sheaves of sets
\[
\pA(\P^1) \to \Kmw_2.
\] Moreover, both sides are homomorphisms of sheaves of groups.
\end{lem}

\begin{proof}
Consider the good pinch maps $c_1 = \cv_D$ and $c_2 = c_+$ of Examples \ref{ex:c+} and \ref{ex:CD}. By definition,
\begin{equation}\label{eq:pADf}
\pA(\sum_D(f_1',\ldots,f_n')) = \pA((\vee f'_i) \circ c_1).
\end{equation} By Proposition \ref{prop:veefic_factors_2_nil}, we have $\pA((\vee f'_i) \circ c_1) = \overline{\pA(\vee f'_i)} \circ( \rho \circ \pA(c_1))$.

By Lemma \ref{lem:pi1deg0} and Proposition \ref{prop:veefic_factors_2_nil}, we have 
\begin{align}\label{eq:pADalgf}
\pA(\sum_{D^\alg}(f_i))) &= \pA((\vee f_i) \circ c_2)\pA(\phi\beta_{m_1,\ldots,m_n,D}) \\ \notag
& =(\overline{\pA(\vee f_i)} \circ( \rho \circ \pA(c_2)) \pA(\phi\beta_{m_1,\ldots,m_n,D}),
\end{align} and the image of $\pA(\phi\beta_{m_1,\ldots,m_n,D})$ lies in $\Kmw_2$. Subtracting \eqref{eq:pADalgf} from \eqref{eq:pADf}, we obtain
\begin{align*}
 &(\overline{\pA(\vee f_i)}\circ \Delta_{c_1,c_2})\pA(\phi\beta_{m_1,\ldots,m_n,D}) =  \pA(\sum_Df_i)(\pA(\sum_{D^\alg}f_i))^{-1}
\end{align*} giving the claimed equality of maps of sheaves of sets. By Lemmas \ref{lem:Deltac12_to_K} and \ref{lem:pi1deg0} the left hand side is a homomorphism of sheaves of groups, showing the claim. 
\end{proof}

\begin{lem}\label{lem:key lemma}
Let $D=\{r_1,\ldots,r_n\}\subset\mb{A}^1_k(k)$ with $r_i \neq r_j$ for $i \neq j$. Suppose we have two $n$-tuples of endomorphisms $f_1,\ldots,f_n\in[\mb{P}^1_k,\mb{P}^1_k]$ and $f_1',\ldots,f_n'\in[\mb{P}^1_k,\mb{P}^1_k]$ such that for each $i$, we have $\rank\deg(f_i)=\rank\deg(f_i').$ If
\[\sum_Df_i=\sum_{D^\alg}f_i,\]
then we have
\[\sum_Df_i'=\sum_{D^\alg}f_i'.\]
\end{lem}

\begin{proof}
Morel shows that $\pA$ induces a group isomorphism
\[[\mb{P}^1_k,\mb{P}^1_k]\cong\mr{End}(\pA(\mb{P}^1)(k))\] as recalled in \eqref{MorelThmP1Anabelian}. See \cite[Section 7.3, Remark~7.32]{Mor12}.  Thus, it is enough to show that 
\[
\pA(\sum_Df_i')=\pA(\sum_{D^\alg}f_i')
\]
Consider the good pinch maps $c_1 = \cv_D$ and $c_2 = c_+$ of Examples \ref{ex:c+} and \ref{ex:CD}. By Lemma~\ref{lem:p1D-pi1Dalg=fiDeltaBeta} 
\begin{align}\label{eq:pAD-pADalg}
 &(\overline{\pA(\vee f'_i)}\circ \Delta_{c_1,c_2})\pA(\phi\beta_{m_1,\ldots,m_n,D}) =  \pA(\sum_Df_i')(\pA(\sum_{D^\alg}f_i'))^{-1}.
\end{align} By Lemmas \ref{lem:Deltac12_to_K} and \ref{overlineveefi_on_K},
\[
(\overline{\pA(\vee f'_i)}\circ \Delta_{c_1,c_2})\pA(\phi\beta_{m_1,\ldots,m_n,D}) = (\overline{\pA(\vee f_i)}\circ \Delta_{c_1,c_2})\pA(\phi\beta_{m_1,\ldots,m_n,D})
\] because $f_i$ also has the rank of its degree equal to $m_i$. Since the equality \eqref{eq:pAD-pADalg} holds with the $f_i$ replacing the $f_i'$, and by hypothesis
\[
\pA(\sum_Df_i) = (\pA(\sum_{D^\alg}f_i)),
\] we must have that 
\[
(\overline{\pA(\vee f'_i)}\circ \Delta_{c_1,c_2})\pA(\phi\beta_{m_1,\ldots,m_n,D}) =0
\] is the trivial map to the identity element in $\Kmw_2$. By \eqref{eq:pAD-pADalg}, it follows that \[
\pA(\sum_Df_i') = \pA(\sum_{D^\alg}f_i'),
\] completing the proof.
\end{proof}

\begin{cor}\label{cor:d=dalg for positive ranks}
Let $D=\{r_1,\ldots,r_n\}\subset\mb{A}^1_k(k)$. Let \[f_1,\ldots,f_n\in[\mb{P}^1_k,\mb{P}^1_k]\]
satisfy $\rank\deg(f_i)>0$ for all $i$. Then 
\[\sum_D(f_1,\ldots, f_n)=\sum_{D^\alg}(f_1,\ldots,f_n).
\]
\end{cor}
\begin{proof}
Let $m_i:=\rank\deg(f_i)$ for each $i$. By Lemma~\ref{lem:key lemma} and Proposition~\ref{prop:local-to-global-rational-function}, it suffices to construct a pointed rational map $f:\mb{P}^1_k\to\mb{P}^1_k$ with vanishing locus $D$ such that $\rank\deg^u_{r_i}(f)=m_i$ for each $i$. The map $f:=\prod_{i=1}^n(x-r_i)^{m_i}$ satisfies these criteria, e.g.~by Lemma~\ref{lem:degree=newton matrix}.
\end{proof}

Let $\delta_{ij}$ denote the Kronecker delta 
\[
\delta_{ij} = \begin{cases}
  1  & i=j \\
  0 & i \neq j.
\end{cases}
\] By a slight abuse of notation, we let $\delta_{ij}\langle 1 \rangle^u$ also denote an endomorphism of $\P^1$ in unstable $\A^1$-homotopy theory with this given degree.

\begin{lem}\label{lem:f_i0+1Delta-fDelta}
Let $c_1$ and $c_2$ be good pinch maps. Fix $i_0 \in \{ 1,\ldots, n\}$. Suppose we have two collections of endomorphisms
\begin{align*}
    f_1,\ldots,f_n&\in[\mb{P}^1_k,\mb{P}^1_k],\\
    f_1',\ldots,f_n'&\in[\mb{P}^1_k,\mb{P}^1_k]
\end{align*}
such that $\rank\deg(f_i)=\rank\deg(f_i')$ for each $i \neq i_0$. Then 
\begin{align*}
(\overline{\pA(\vee (f_i + \delta_{i_0 i} \langle 1 \rangle^u))}\circ \Delta_{c_1,c_2}) - (\overline{\pA(\vee (f_i ))}\circ \Delta_{c_1,c_2}) &=\\
(\overline{\pA(\vee (f'_i + \delta_{i_0 i} \langle 1 \rangle^u))}\circ \Delta_{c_1,c_2}) - (\overline{\pA(\vee (f'_i ))}\circ \Delta_{c_1,c_2})&
\end{align*}
are equal maps
\[
\pA(\P^1) \to \Kmw_2.
\]
\end{lem}

\begin{proof}
As above in Lemma~\ref{lem:wppabeta}, for an integer $m$, let $m: \G_m \to \G_m$ denote the map $z \mapsto z^m$. Consider maps
\[
m_i \otimes m_j: \G_m \otimes \G_m \to \G_m \otimes \G_m
\]  where $m_i$ and $m_j$ are integers. Then there is an equality $(m_i + 1) \otimes m_j = m_i \otimes m_j + 1 \otimes m_j$ of such maps. Let $m_i = \rank\deg(f_i)$ for $i=1,\ldots,n$. Then we have an equality 
\[
\times_{i \neq j} ((m_i + \delta_{i_0 i} )\otimes m_j) = \times_{i \neq j} (m_i \otimes m_j) + \times_{j = 1, j \neq i_0}^n 1 \otimes m_j
\] of maps
\[
\times_{i \neq j} \G_m \otimes \G_m \to \times_{i \neq j} \G_m \otimes \G_m
\]

Let $\sigma: \times_{i \neq j} a(\pA(\P^1)^{\ab} \otimes \pA(\P^1)^{\ab}) \to K$ denote the epimorphism of Lemma~\ref{lem:Kr_central}. By Lemma \ref{overlineveefi_on_K}, it follows that
\begin{equation}\label{eq:veef_i+1_restricted_K}
    \overline{\pA(\vee_i (f_i + \delta_{i_0 i} \langle 1 \rangle^u) )} \circ \sigma = \overline{\pA(\vee_i f_i)}\circ \sigma + (\times_{i \neq j} \commu) \circ \times_{j = 1, j \neq i_0}^n (1 \otimes m_j) \circ \times_{j = 1, j \neq i_0}^n(\wp \otimes \wp)
\end{equation}

Note that 
\[
(\times_{i \neq j} \commu) \circ \times_{j = 1, j \neq i_0}^n (1 \otimes m_j) \circ \times_{j = 1, j \neq i_0}^n(\wp \otimes \wp)
\] only depends on $m_1,\ldots, m_{i_0 -1}, m_{i_0 + 1}, \ldots, m_n$ and in particular is independent of $m_{i_0}$. By Lemma~\ref{lem:Deltac12_to_K}  $\Delta_{c_1,c_2}$ determines a homomorphism
\begin{equation*}
\Delta_{c_1,c_2}: \pA(\P^1) \to K
\end{equation*} By Equation~\eqref{eq:veef_i+1_restricted_K} and the fact that $\sigma$ is an epimorphism, it follows that 
\[
\overline{\pA(\vee_i (f_i + \delta_{i_0 i} \langle 1 \rangle^u) )} \circ \Delta_{c_1,c_2} - \overline{\pA(\vee_i f_i )} \circ \Delta_{c_1,c_2} 
\] only depends on $m_1,\ldots, m_{i_0 -1}, m_{i_0 + 1}, \ldots, m_n$.
\end{proof}

\begin{proof}[Proof of Theorem~\ref{thm:D-sum-all-f_i}]
 We prove that $\pA(\sum_Df_i)=\pA(\sum_{D^\alg}f_i)$ which is sufficient by Morel's theorem that $\P^1$ is $\A^1$-anabelian \cite[Section 7.3, Remark~7.32]{Mor12} as recalled in \eqref{MorelThmP1Anabelian}. By Corollary~\ref{cor:d=dalg for positive ranks}, we have $\sum_Df_i=\sum_{D^\alg}f_i$ whenever each $\deg^u(f_i)$ has positive rank. For the inductive hypothesis, assume that we have $(m_1,\ldots,m_n)\in\mb{Z}^n$ such that $\sum_Df_i=\sum_{D^\alg}f_i$ whenever $\rank\deg^u(f_i)\geq m_i$ for each $i$.

We wish to show that if $f_1,\ldots,f_n\in[\mb{P}^1_k,\mb{P}^1_k]$ satisfy $\rank\deg^u(f_{i_0})=m_{i_0}-1$ for some $1\leq i_0\leq n$ and $\rank\deg^u(f_i)=m_i$ for $i\neq i_0$, then $\sum_Df_i=\sum_{D^\alg}f_i$. Let $g_{i_0}\in[\mb{P}^1_k,\mb{P}^1_k]$ be an endomorphism such that $\deg^u(g_{i_0})=\deg^u(f_{i_0})+\langle 1\rangle^u$, which exists since $\GW^u(k)\cong[\mb{P}^1_k,\mb{P}^1_k]$. Let $g_i:=f_i$ for all $i\neq i_0$. By our inductive hypothesis, we have
\begin{equation}\label{eq:Dg=Dalg_g}
\sum_Dg_i=\sum_{D^\alg}g_i.
\end{equation}
By Lemma~\ref{lem:p1D-pi1Dalg=fiDeltaBeta}, we have that
\begin{equation}\label{g_ibeta=0}
(\overline{\pA(\vee g_i)} \circ \Delta_{\cv_D,c_+})\pA(\phi\beta_{m_1,\ldots,m_{i_0}+1,\ldots,m_n,D}) = 0.
\end{equation} 
 By the same reasoning
\begin{equation}\label{g_i+1beta=0}
(\overline{\pA(\vee (g_i + \delta_{i i_0} \langle 1 \rangle^u))}\circ \Delta_{\cv_D,c_+})\pA(\phi\beta_{m_1,\ldots,m_{i_0}+2,\ldots,m_n,D}) = 0.
\end{equation} 
 By Lemma \ref{lem:f_i0+1Delta-fDelta}, for all $n$-tuples $f_1',\ldots,f_n'$ with $\rank f_i' = m_i$ for $i \neq i_0$, we have
\begin{align*}
(\overline{\pA(\vee (f'_i + \delta_{i_0 i} \langle 1 \rangle^u))}\circ \Delta_{c_1,c_2}) - (\overline{\pA(\vee (f'_i ))}\circ \Delta_{c_1,c_2})& =
\\(\overline{\pA(\vee (g_i + \delta_{i_0 i} \langle 1 \rangle^u))}\circ \Delta_{c_1,c_2}) - (\overline{\pA(\vee (g_i ))}\circ \Delta_{c_1,c_2}).
\end{align*}

By Equations \eqref{g_ibeta=0} and \eqref{g_i+1beta=0},
\begin{align*}
(\overline{\pA(\vee (g_i + \delta_{i_0 i} \langle 1 \rangle^u))}\circ \Delta_{c_1,c_2}) - (\overline{\pA(\vee (g_i ))}\circ \Delta_{c_1,c_2})&= 
\\ - \pA(\phi\beta_{m_1,\ldots,m_{i_0}+2,\ldots,m_n,D}) + \pA(\phi\beta_{m_1,\ldots,m_{i_0}+1,\ldots,m_n,D})&.
\end{align*}

By Lemma \ref{lem:pi1deg0},
\begin{align*}
-\pA(\phi\beta_{m_1,\ldots,m_{i_0}+2,\ldots,m_n,D}) + \pA(\phi\beta_{m_1,\ldots,m_{i_0}+1,\ldots,m_n,D})&= \\
\pA(-\phi\beta_{m_1,\ldots,m_{i_0}+2,\ldots,m_n,D} + \phi\beta_{m_1,\ldots,m_{i_0}+1,\ldots,m_n,D})&.
\end{align*} 

By direct calculation, 
\begin{align}\label{eq:rscalc}
-\beta_{m_1,\ldots,m_{i_0}+2,\ldots,m_n,D} + \beta_{m_1,\ldots,m_{i_0}+1,\ldots,m_n,D} = \\ \notag  
\frac{\prod_{\substack{i \neq j \\ i,j \neq i_0 }} (r_i - r_j)^{2 m_i m_j} \prod_{j \neq i_0} (r_{i_0}-r_j)^{2 (m_{i_0} + 1) m_j}}{\prod_{\substack{i \neq j \\ i,j \neq i_0 }} (r_i - r_j)^{2 m_i m_j}\prod_{j \neq i_0} (r_{i_0}-r_j)^{2 (m_{i_0} + 2 ) m_j}} = \\ \notag
\prod_{j \neq i_0} (r_{i_0}-r_j)^{-2 m_j} = \\ \notag
-\beta_{m_1,\ldots,m_{i_0}+1,\ldots,m_n,D} + \beta_{m_1,\ldots,m_{i_0},\ldots,m_n,D}
\end{align} 

Combining, we see that for all $n$-tuples $f_1',\ldots,f_n'$ with $\rank f_i' = m_i$ for $i \neq i_0$, we have
\begin{align}\label{eq:f'_i+1-f_i'=r}
(\overline{\pA(\vee (f'_i + \delta_{i_0 i} \langle 1 \rangle^u))}\circ \Delta_{c_1,c_2}) - (\overline{\pA(\vee f'_i )}\circ \Delta_{c_1,c_2}) = \\ \notag
\pA((0,\prod_{j \neq i_0} (r_{i_0}-r_j)^{-2 m_j})).
\end{align}

Now take $f_i' = f_i$ . Then $f_i' + \delta_{i_0 i} \langle 1 \rangle^u = g_i$. By Equations~\eqref{g_ibeta=0} and \eqref{eq:f'_i+1-f_i'=r}, 
\begin{align*}
(\overline{\pA(\vee (f_i ))}\circ \Delta_{c_1,c_2})& = 
-\pA(\phi\beta_{m_1,\ldots,m_{i_0}+1,\ldots,m_n,D})-\pA((0,\prod_{j \neq i_0} (r_{i_0}-r_j)^{-2 m_j})) 
\\&= -\pA(\phi\beta_{m_1,\ldots,m_{i_0},\ldots,m_n,D})
\end{align*} where the last equality follows from \eqref{eq:rscalc} and Lemma \ref{lem:pi1deg0}. 
By Lemma~\ref{lem:p1D-pi1Dalg=fiDeltaBeta}, we have
\[
\pA(\sum_Df_i)=\pA(\sum_{D^\alg}f_i)
\] as desired.
\end{proof}

\appendix
\section{Code}\label{sec:code}
Here is some code that calculates duplicants and the square root of the ordinary discriminant. We used this code to conjecture a closed formula for the duplicant, which we then proved in Theorem~\ref{thm:duplicant}.

\begin{lstlisting}
def coefficients(f,N):
    # deal with constant term
    coeffs = [f.subs(x=0)]
    # append other coefficients
    for i in range(1,N):
        coeffs.append(f.coefficient(x^i))
    return(coeffs)

def vand(n): # compute sqrt(disc(r0,...,r(n-1)))
    r = var('r',n=n)
    factors = []
    for i in range(n):
        for j in range(i+1,n):
            factors.append(r[i]-r[j])
    return(prod(factors))

def dupl(n,e): # compute duplicant
    x = var('x')
    r = var('r',n=n)
    N = sum(e)
    coeff_list = []
    
    for i in range(n):
        for j in range(1,e[i]+1):
            # generate f/(x-r_i)^j
            e_new = e.copy()
            e_new[i] = e[i]-j
            f = prod([(x-r[l])^e_new[l]\
                      for l in range(n)]).expand()
            coeff_list.append(coefficients(f,N))
            
    # compute det^2 of coefficient matrix
    coeff_matrix = matrix(coeff_list)
    return(coeff_matrix.det()^2)
\end{lstlisting}

\bibliography{unstable-degree}{}
\bibliographystyle{alpha}
\end{document}